\documentclass [twoside,reqno,12pt] {amsart}

\usepackage{amsfonts}
\usepackage{amssymb}
\usepackage{a4}
\usepackage [all] {xy}

\usepackage{hyperref}

\newtheorem{thm}{Theorem}[section]
\newtheorem{cor}[thm]{Corollary}
\newtheorem{lem}[thm]{Lemma}
\newtheorem{prop}[thm]{Proposition}

\newtheorem{rem}[thm]{Remark}

\numberwithin{equation}{section}

\newcommand{\N}{\mathbb{N}}

\newcommand{\R}{\mathbb{R}}
\newcommand{\Q}{\mathbb{Q}}

\newcommand{\supp}{\operatorname{supp}}

\newcommand{\Z}{\mathbb{Z}}

\def\hat{\widehat}
\def\tilde{\widetilde}
\def \bfo {\begin {eqnarray*} }
\def \efo {\end {eqnarray*} }
\def \ba {\begin {eqnarray*} }
\def \ea {\end {eqnarray*} }
\def \beq {\begin {eqnarray}}
\def \eeq {\end {eqnarray}}
\def \supp {\hbox{supp }}

\def \p {\partial}

\def\hat{\widehat}
\def\tilde{\widetilde}
\def \bfo {\begin {eqnarray*} }
\def \efo {\end {eqnarray*} }
\def \ba {\begin {eqnarray*} }
\def \ea {\end {eqnarray*} }
\def \beq {\begin {eqnarray}}
\def \eeq {\end {eqnarray}}
\def \supp {\hbox{supp}}

\def \p {\partial}

\begin{document}

 \title[An inverse problem for a hyperbolic system on a vector bundle ]{An inverse problem for a hyperbolic system on a vector bundle and
energy measurements}

\author[Krupchyk]{Katsiaryna Krupchyk}

\address
        {K. Krupchyk, Department of Mathematics and Statistics \\
         University of Helsinki\\
         P.O. Box 68 \\
         FI-00014   Helsinki\\
         Finland}

\email{katya.krupchyk@helsinki.fi}

\author[Lassas]{Matti Lassas}

\address
        {M. Lassas, Department of Mathematics and Statistics \\
         University of Helsinki\\
         P.O. Box 68 \\
         FI-00014   Helsinki\\
         Finland}

\email{matti.lassas@helsinki.fi}

\date{}



\maketitle

\begin{abstract} 
A uniqueness result in the inverse problem for an inhomogeneous hyperbolic system on a real vector bundle over a smooth compact manifold, based on energy measurements for improperly known sources, is established. 
\end{abstract}

\section{Introduction and statement of result}

\label{sec_intr_1}

Let $M$ be a compact smooth manifold of dimension $n\ge 2$ without boundary,  and let $V$ be a smooth real vector bundle with a Riemannian structure over $M$. 
Assume that a smooth positive density $d\mu$ is given on $M$. 
Let $A:C^\infty(M,V)\to C^\infty(M,V)$ 
be  an elliptic formally self-adjoint second-order partial differential operator, acting on smooth sections of $V$.  
We assume that $A$ is formally positive, i.e., 
\begin{equation}
\label{eq_int_0}
(Au,u)_{L^2(M,V)}\ge c(u,u)_{L^2(M,V)},\quad u\in C^\infty(M,V),
\end{equation}
where $c>0$ and 
\[
(u,v)_{L^2(M,V)}=\int_M \langle u(x),v(x) \rangle_xd\mu(x),
\]
with $\langle\cdot,\cdot\rangle_x$ being the inner product in the fiber $\pi^{-1}(x)$ of the vector bundle $V$, assigned by the Riemannian structure.
Viewed as an unbounded operator on $L^2(M,V)$ and equipped with the domain $\mathcal{D}(A)=H^2(M,V)$,  the operator $A$ becomes self-adjoint  with a positive lower bound $c$ and a real discrete spectrum, see \cite{Shubin_book}.

 Let $F\in L^1(\R,L^2(M,V))$ with $\supp_t(F)\subset (-\tau,+\infty)$ for some large $\tau=\tau(F)>0$.  Here $\supp_t (F)$ stands the time support of $F$. 
Associated with the operator $A$, we consider the following hyperbolic initial value problem
\begin{equation}
\label{eq_int_1}
\begin{cases}(\partial_t^2+A(x,\p_x))u(x,t)=F(x,t),\quad x\in M, \quad t> -\tau,\\
u(x,t)|_{t=-\tau}=0,\quad \partial_t u(x,t)|_{t=-\tau}=0.
\end{cases}
\end{equation}
The problem \eqref{eq_int_1} has a unique solution
\[
u=u^F\in C^1(\R,L^2(M,V))\cap C^0(\R,H^1(M,V)), 
\]
see 
\cite{Las_Lions_Trig}. 
The energy of the wave $u^F$ produced by the source $F$ is defined by the formula, 
\begin{equation}
\label{eq_int_2}
E_A(F;t)=\frac{1}{2}\int_M (\langle\partial_t u^F(x,t),\partial_t u^F(x,t)\rangle_x+\langle Au^F(x,t),u^F(x,t)\rangle_x)d\mu(x).
\end{equation}
 As the operator $A$ is self-adjoint, the energy 
is conserved, i.e  if $F=0$ for all $t\ge t_0$ then $E_A(F,t)=E_A(F,t_0)$, c.f.  \eqref{eq_conservation_energy} below.  

For $T\ge 0$,  we denote by 
\[
(\tau_T F)(x,t)=F(x,t+T),
\]
the time shift of a source $F$.

Let $\mathcal{B}_M=\{U_p, p=1,2,\dots\}$ be a basis for topology of $M$ such that $U_1=M$, and let $\mathcal{B}_{\R}=\{I_q, q=1,2,\dots\}$ be  a basis for topology  of $\R$ consisting of open bounded intervals.  For any $U_p\in \mathcal{B}_M$, $p\ge 2$,  and any $I_q\in\mathcal{B}_{\R}$, $q=1,2,\dots$, we assume that there is a countable set of sources 
$
\mathcal{F}_{V;U_p,I_q}\subset C^0(\overline{I_{q}},L^2(U_p,V))$, 
which is dense in $L^2(I_{q},L^2(U_p,V))$. Furthermore, corresponding
to the case when $p=1$, we assume that for any $I_q\in\mathcal{B}_{\R}$, $q=1,2,\dots$,   there is a countable set of sources
$\mathcal{F}_{V;M,I_q}\subset C^\infty_0(\overline{I_q},C^\infty(M, V))$,
which is dense in $C^\infty_0(\overline{I_q},C^\infty(M, V))$. Here 
\[
C^\infty_0(\overline{I_q},C^\infty(M, V))=\{F\in C^\infty(\R,C^\infty(M, V)),\supp_t(F)\subset \overline{I_q}\}.
\]
This assumption is generic in the sense that such countable families of sources  
can be almost surely generated by taking some sequences of realizations of suitable independent identically distributed Gaussian random variables, see Section \ref{sec_random_sources}. 

As the sets $\mathcal{F}_{V;U_p,I_q}$ are countable, we shall enumerate its elements, 
\[
\mathcal{F}_{V; U_p,I_q}=\{F^{p,q}_i:i=1,2,\dots\}, \quad p,q=1,2,\dots,
\]
and fix these enumerations.

For any $N\in \N$, $p_1,\dots,p_N$, $q_1,\dots,q_N$,  $i_1,\dots, i_N\in \N$, $0\le T_1,\dots, T_N\in \Q$, and $0\le t\in \Q$, such that the time supports of all time shifts $\tau_{T_k}F^{p_k,q_k}_{i_k}$ of the sources $F^{p_k,q_k}_{i_k}\in \mathcal{F}_{V;U_{p_k},I_{q_k}}$, $k=1,\dots,N$, are pairwise disjoint,  we introduce the energy functions
\begin{equation}
\label{eq_int_4}
\mathcal{E}^{(N)}_{V,A}:\bigg((p_1,q_1,i_1),\dots,(p_N,q_N, i_N), T_1,\dots, T_N; t\bigg)\mapsto E_A\bigg( \sum_{k=1}^N \tau_{T_k}F^{p_k,q_k}_{i_k}; t\bigg),
\end{equation}
where the energy $E_A$ is given by \eqref{eq_int_2}.

We shall assume that the manifold $M$ together with its  topological and differential structures, as well as a smooth positive density $d\mu$ on it are known.  
The purpose of this paper is to show that  given the energy functions $\mathcal{E}^{(N)}_{V,A}$ for all $N=1,2,\dots$, we can recover the vector bundle $V$ with a Riemannian  structure  over $M$, as well as the elliptic operator $A$.  By the determination of a vector bundle we understand the determination of an isomorphic copy of it. 

We would like to emphasize that the sources $F^{p_k,q_k}_{i_k}\in \mathcal{F}_{V;U_{p_k},I_{q_k}}$, which take values in an unknown vector bundle $V$,  are not given  in our problem.  However,   the knowledge of the energy functions $\mathcal{E}^{(N)}_{V,A}$ implies that together with the value of the energy  $E_A\bigg( \sum_{k=1}^N \tau_{T_k}F^{p_k,q_k}_{i_k}; t\bigg)$, we know the indices $i_k$, which enumerate  the sources $F^{p_k,q_k}_{i_k}$ in the set $\mathcal{F}_{V;U_{p_k},I_{q_k}}$, as well as the indices $p_k$ and $q_k$, which give us the information about the support of the source $F^{p_k,q_k}_{i_k}$, i.e.  $\supp(F^{p_k,q_k}_{i_k})\subset \overline{U_{p_k}}\times \overline{I_{q_k}} $,  the time shifts $T_1, \dots, T_N$, and the time $t$, which correspond to the measured energy.     

The main result of this paper is as follows. 

\begin{thm}  
\label{thm_main}

Let $M$ be a  compact smooth manifold of dimension $n\ge 2$ without boundary and let $d\mu$ be a smooth positive density on $M$.  Let $V_1$ and $V_2$ be real smooth vector bundles with Riemannian structures over $M$, and  let $A_j:C^\infty(M,V_j)\to C^\infty(M,V_j)$ 
be  an elliptic formally self-adjoint positive second-order partial differential operator on sections of $V_j$, $j=1,2$. 
Assume that for $V_j$, $j=1,2$,  and any $U_p\in\mathcal{B}_M$, $p\ge 2$, and any $I_q\in \mathcal{B}_{\R}$, $q=1,2,\dots$, 
there is a countable set of sources
$
\mathcal{F}_{V_j;U_p,I_q}\subset C^0(\overline{I_{q}},L^2(U_p,V_j))$, 
which is dense in $L^2(I_{q},L^2(U_p,V_j))$. 
Furthermore, for $V_j$, $j=1,2$,  and any $I_q\in \mathcal{B}_{\R}$, $q=1,2,\dots$, we assume that there is a countable set of sources
$\mathcal{F}_{V_j;M,I_q}\subset C_0^\infty(\overline{I_q},C^\infty(M, V_j))$,
which is dense in $C_0^\infty(\overline{I_q},C^\infty(M, V_j))$. 
If there are enumerations of the elements of the sets $\mathcal{F}_{V_1;U_p,I_q}$ and  $\mathcal{F}_{V_2;U_p,I_q}$, such that for the energy functions $\mathcal{E}^{(N)}_{V_1, A_1}$ and $\mathcal{E}^{(N)}_{V_2, A_2}$,  introduced in \eqref{eq_int_4},  we have  
\begin{equation}
\label{eq_int_5}
\mathcal{E}^{(N)}_{V_1, A_1}=\mathcal{E}^{(N)}_{V_2,A_2},\quad \textrm{for all}\ N=1,2,\dots,
\end{equation}
 then there is an isometry $\Phi:V_1\to V_2$ such that $A_2=\Phi A_1 \Phi^{-1}$. 
\end{thm}

\textbf{Remark.}
The equality  \eqref{eq_int_5} of the functions $\mathcal{E}^{(N)}_{V_j, A_j}$  signifies in particular that the domains of their definition agree.

The main motivation to study such kind of inverse problems comes from the fact that in many practical situations one can only measure the energy corresponding to some sources, but the exact form of the sources is not known.   Such poorly known sources are typical in physical problems of various scales. As a first example, let us mention that the precise parameters describing collisions of elementary particles in atomic physics are usually not known, and one just observes that the collision has occurred. On the macroscopic scale, electrodes used in electrical impedance tomography have unknown contact impedances. Furthermore, in oil exploration, one uses explosives as sources,  and hence the source is difficult to model mathematically. As the final example, for underwater air guns in imaging of the sea bottom, modeling the sources precisely is quite difficult. However, in all these instances, one has a good knowledge of the support of the source.

Theorem \ref{thm_main} can be useful in the following practical situation. Assume that in some physical experiment, one produces a suitable dense set of sources $\mathcal{F}_{V;U_{p_k},I_{q_k}}$, and measures the energy of the waves $u^F$, corresponding to the sources
\begin{equation}
\label{eq_sor_*}
F= \sum_{k=1}^N \tau_{T_k}F^{p_k,q_k}_{i_k},\quad F^{p_k,q_k}_{i_k}\in \mathcal{F}_{V;U_{p_k},I_{q_k}},
\end{equation}
such that the time supports of all time shifts $\tau_{T_k}F^{p_k,q_k}_{i_k}$ of the sources $F^{p_k,q_k}_{i_k}$ are pairwise disjoint.  One encodes the results of these measurements 
into the energy functions $\mathcal{E}^{(N)}_{V,A}$, $N=1,2,\dots$.  Our result shows that  given only the energy functions  $\mathcal{E}^{N}_{V,A}$, $N=1,2,\dots$,  we are able to determine the physical model (i.e., the vector bundle $V$ with the Riemannian  structure, and the operator $A$).

We shall now make a few comments, regarding the form of the sources \eqref{eq_sor_*}, for which our measurements are performed.  First,  our assumption that the energy measurement for a sum $F_1+F_2$ 
cannot be done unless the time supports of $F_1$ and $F_2$ are
disjoint is  motivated by  practical applications.  Indeed,  
if
$F_1$ and $F_2$ are sources produced by some devices $D_1$ and $D_2$, we cannot
usually put the device $D_1$ in the same location which $D_2$ occupies at the same time.
For instance, in  soil imaging we cannot implement two explosions at the same place
with a very small time difference.

The second fact that we would like to emphasize  is  that the data for our inverse problem 
does not contain 
energies corresponding to the sources 
\begin{equation}
\label{eq_sor_*2}
\sum_{k=1}^N a_{k} \tau_{T_k}F^{p_k,q_k}_{i_k},\quad a_k\in \R,
\end{equation}
i.e., arbitrary linear combinations of the sources $F^{p_k,q_k}_{i_k}$ and their time shifts.   This is again motivated by practical applications. Indeed,  if a source $F^{p_k,q_k}_{i_k}$ is poorly  
known, it may be difficult to implement a measurement of the form $aF^{p_k,q_k}_{i_k}$ with $a\in \R$.  For example,  for
 the wave equation 
\[
(\p_t^2-c(x)^2\Delta)u(x,t)=F(x,t),
\]
modeling the acoustic pressure $u(x,t)$, 
it may be easy to implement measurements with a source $F$ corresponding
to a local increase of the pressure, e.g. an explosion,
but difficult to implement the source $-F$ (an "anti-explosion"). 
However, in order to determine the vector bundle we have to find  energies, which correspond to the sources \eqref{eq_sor_*2}, at suitable times,  using our data. This requires careful considerations, which are carried out in Section \ref{sec_meas_conclusion}.

Similar problems of  recovering an unknown operator  and a Riemannian manifold  from energy measurements
have been encountered in inverse boundary value problems. 
The prototypical inverse problem here is the inverse conductivity problem, called also the  Calder\'on inverse problem (see \cite{Cal}).
Consider the conductivity equation
\beq\label{eq: conductivity}
\nabla\cdotp\sigma(x)\nabla u(x)&=&0\quad\hbox{in }D,\\
u|_{\p D}&=&f, \nonumber
\eeq
where $D\subset \R^n$ is a smooth bounded domain, $\sigma(x)$ is a positive-definite
 matrix corresponding to the conductivity at the point $x$, and
$u$ is the electric potential having boundary value $f$. The inverse conductivity problem consists of 
the determination of $\sigma$ from measurements  done at the boundary $\p D$.
One possibility to define the boundary measurements is the following: assume that for all boundary values $f$,
we measure
\[
Q_\sigma(f)=\int_{D} (\sigma(x)\nabla u(x))\cdotp \nabla u(x)\,dx.
\]
This is equivalent to measuring the power needed to keep the voltage $u$ at
the boundary $\p D$ to be equal to $f(x)$. Physically, the boundary 
value $u|_{\p D}=f$ is produced by electrodes attached on the boundary of the body,
and $Q_\sigma(f)$ corresponds
to the power of the heat produced by the caused currents in the domain $D$. Because of the conservation
of energy, the power produced by the heat is equal to the power needed to keep
the electrodes at the specified voltages. When $\sigma$ is equal to a scalar function  times
the identity matrix, the conductivity is called isotropic. In this case the 
measurements are known to determine the conductivity uniquely. This was established for sufficiently regular conductivities in \cite{syluhl87, nachman96}.   See also \cite{astpaiv06, Bukhgeim08, ImanUhlMas, KenSjostrandUhl07} for some additional important contributions.

 When the conductivity $\sigma$ in (\ref{eq: conductivity}) is matrix
 valued, i.e. anisotropic,  one knows  that
the conductivity $\sigma$ cannot be determined uniquely and that the best one can
hope for is  to determine  $\sigma$ up to a diffeomorphism of the domain $D$, which preserves the boundary. This has been shown
in dimension $n= 2$ in \cite{AstPaivLass05, sunuhl03, Syl90}. In dimensions $n\geq 3$,  determining the conductivity
is equivalent to determining the isometry type of a smooth Riemannian metric corresponding to the
conductivity, which constitutes a longstanding open problem. 
For related works, 
see \cite{GrKurLassUhl09, HM08, KKL, LTU03, LU01}.

As another example of an inverse problem on a Riemannian manifold which motivates our study, let us  consider the wave equation, 
\beq
\label{eq: scalar_wave}
\begin{cases}(\partial_t^2-\Delta_g)u(x,t)=0\quad (x,t)\in M\times \R_+,\\
u|_{t=0}=0,\quad \partial_t u|_{t=0}=0, \\
u|_{\p M\times \R_+}=f,
\end{cases}
\eeq
on a compact Riemannian manifold $(M,g)$ with boundary $\p M$. Assume that we know 
the boundary $\p M$ and can measure for all $f\in C^\infty_0(\p M\times \R_+)$,
 the final energy $Q_g(f)$ 
of the wave $u$ produced by $f$, 
that is,
\ba
Q_g(f)=\lim_{t\to \infty} \frac 12\int_M (|\nabla_g u(x,t)|_g^2+|\p_t u(u,t)|^2)d\mu_g(x).
\ea
By energy conservation, this can be considered as the total energy needed
to force the boundary value $u|_{\p M\times \R_+}$ to be equal to $f$.
When $Q_g(f)$ is known for all $f\in C^\infty_0(\p M\times \R_+)$, then
 the isometry type of the Riemannian manifold $(M,g)$ can be recoved  \cite{KKLM04}, see also \cite{AKKLT04, KKL}.

This paper can be viewed as an extension of the class of inverse problems based on energy measurements to the setting of hyperbolic operators on general vector bundles.   Studying inverse problems in the geometric  context of vector bundles is  both natural and important, as for instance, in relativistic quantum mechanics and quantum field theory, phenomena are often modeled by PDE on sections of vector bundles.  The investigations in this work are further complicated by the fact that our energy measurements are performed for improperly known sources, which, as we argued above, is relevant for many practical applications.

Finally, we would like to mention the paper \cite{KL-Dirac},  where an  inverse problem for the Dirac equation on a vector bundle over a compact Riemannian manifold with nonempty boundary is investigated.

The paper is organized  as follows. In Section \ref{sec_definitions} we present some basic facts concerning real vector bundles and hyperbolic systems.  
Section \ref{sec_main_result} is devoted to the proof of Theorem \ref{thm_main}.  Finally, in Section \ref{sec_random_sources} we exploit probabilistic arguments to show that our assumptions on the existence of countable dense sets of sources is generic, in the sense that the latter can be realized almost surely, by taking sequences of independent identically distributed Gaussian random variables.

\section{Preliminaries}

\label{sec_definitions}

\subsection{Some facts on real  vector bundles}

Let $M$ be a compact smooth manifold of dimension $n\ge 2$ without boundary and let $V$ be a smooth real vector bundle over $M$. We denote by $\pi:V\to M$ the projection onto the base manifold. Each fiber $\pi^{-1}(x)$ has a structure of a real vector space isomorphic to  $\mathbb{R}^d$ and carries  an inner product $\langle\cdot,\cdot\rangle_x$, smoothly depending on $x$, given by the  Riemannian structure.

Given two vector bundles $\pi_1:V_1\to M$ and $\pi_2:V_2\to M$, a smooth map
$\Phi:V_1\to V_2$ is a vector bundle homomorphism if $\pi_1=\pi_2\circ\Phi$ and 
$\Phi$ restricts to a linear map $\pi_1^{-1}(x)\to \pi_2^{-1}(x)$ on each fiber. 
A bundle homomorphism $\Phi:V_1\to V_2$ with an inverse which is also a bundle homomorphism is called a vector bundle isomorphism, and then $V_1$ and $V_2$ are said to be isomorphic vector bundles. If $\Phi$ preserves the Riemannian structure, then $\Phi$ is called an isometry. 

Let $\{(U_\alpha,\phi_\alpha)\}$ be an atlas of a vector bundle $V$, i.e.  $\{U_\alpha\}$ is an open cover of $M$ and $\phi_\alpha:\pi^{-1}(U_\alpha)\to U_\alpha\times \mathbb{R}^d$ are local trivializations.
For any two vector bundle charts $(U_\alpha,\phi_\alpha)$ and $(U_\beta,\phi_\beta)$ such that $U_\alpha\cap U_\beta\ne\emptyset$, the composition $\phi_\alpha\circ\phi^{-1}_\beta:(U_\alpha\cap U_\beta)\times \mathbb{R}^d\to (U_\alpha\cap U_\beta)\times \mathbb{R}^d$  must be of the form $\phi_\alpha\circ\phi^{-1}_\beta(x,v)=(x,t_{\alpha\beta}(x)v)$ for some smooth map $t_{\alpha\beta}:U_\alpha\cap U_\beta\to \textrm{GL}(d,\mathbb{R})$ which is called a transition map. For a given vector bundle atlas the family of transition maps always satisfies the following properties:

\begin{equation}
\label{eq_transition}
\begin{aligned}
t_{\alpha\alpha}(x)&=\text{Id}\text{ for all }x\in U_\alpha\text{ and all }\alpha;\\
t_{\alpha\beta}(x)\circ t_{\beta\alpha}(x)&=\text{Id}\text{ for all }x\in U_\alpha\cap U_\beta;\\
t_{\alpha\gamma}(x)\circ t_{\gamma\beta}(x)\circ t_{\beta\alpha}(x)&=\text{Id}\text{ for all }x\in U_\alpha\cap U_\beta\cap U_\gamma.
\end{aligned}
\end{equation}

A family of maps $\{t_{\alpha\beta}\}$, $t_{\alpha\beta}:U_\alpha\cap U_\beta\to \textrm{GL}(d,\mathbb{R})$ which satisfies \eqref{eq_transition} for some cover of $M$ is called a $\textrm{GL}(d,\mathbb{R})$-\emph{cocycle}.

The following theorem  gives the minimal information  required to reconstruct a vector bundle up to an isomorphism and it will be used in the proof of Theorem \ref{thm_main}. 
\begin{thm}{ \cite[Section 9.2.2]{nakahara}}
\label{thm_geometry1}
Given a manifold $M$, a cover $\{U_\alpha\}$ of $M$ and  a $\textrm{GL}(d,\mathbb{R})$-cocycle $\{t_{\alpha\beta}\}$
for the cover, there exists a vector bundle with an atlas $\{(U_\alpha,\phi_ \alpha)\}$ satisfying $\phi_\alpha\circ\phi^{-1}_\beta(x,v)=(x,t_{\alpha\beta}(x)v)$ on nonempty overlaps $U_\alpha\cap U_\beta$. A vector bundle constructed in such a way is unique up to an isomorphism. 
\end{thm}

\subsection{Some facts on hyperbolic systems}

Let $M$ be a compact smooth manifold of dimension $n\ge 2$ without boundary,  and let $V$ be a smooth real vector bundle with a Riemannian structure over $M$. 

Let $A:C^\infty(M,V)\to C^\infty(M,V)$ 
be  an elliptic formally self-adjoint  second-order partial differential operator, acting on smooth sections of $V$.  
We assume that $A$ is formally positive in the sense of \eqref{eq_int_0}. 
Viewed as an unbounded operator on $L^2(M,V)$ and equipped with the domain $\mathcal{D}(A)=H^2(M,V)$,  the operator $A$ becomes self-adjoint  with a positive lower bound $c$,
\[
(Au,u)_{L^2(M,V)}\ge c\|u\|_{L^2(M,V)}^2,\quad u\in \mathcal{D}(A), 
\]
 and  a real discrete spectrum, see \cite{Shubin_book}.  

In fact the ellipticity shows that this bound can be strengthened to the following estimate,
\begin{equation}
\label{eq_3_1_1}
(Au,u)_{L^2(M,V)}\ge c_0\|u\|_{H^1(M,V)}^2,\quad u\in \mathcal{D}(A), \quad c_0>0,
\end{equation}
see \cite{Shubin_book}.  For the upper bound we have
\begin{equation}
\label{eq_3_1_2}
(Au,u)_{L^2(M,V)}=\|A^{1/2}u\|_{L^2(M,V)}^2\le C\|u\|_{H^1(M,V)}^2, \quad u\in \mathcal{D}(A),\quad  C>0,
\end{equation}
since $A^{1/2}$ is an  elliptic pseudodifferential operator of order one, see \cite{Shubin_book}. 

Furthermore, we have $\mathcal{D}(A^k)=H^{2k}(M,V)$, $k=1,2,\dots$, where $\mathcal{D}(A^k)$ is considered as a Hilbert space with the graph norm 
$(\|u\|^2_{L^2(M,V)}+\|A^k u\|^2_{L^2(M,V)})^{1/2}$, see \cite{KKL}. This together with the fact that $A^k$ is positive implies that 
\begin{equation}
 \label{eq_norms_1}
\|A^k u\|_{L^2(M,V)}^2\asymp \|u\|_{H^{2k}(M,V)}^2, \quad u\in \mathcal{D}(A^k), \quad k=1,2,\dots.
\end{equation}
Here the notation  $a \asymp b$ states that    $C_1a  \le b\le C_2a$ for some constants $C_1,C_2>0$.

We also have for $k=0,1,2,\dots$,
\begin{equation}
 \label{eq_norms_2}
  ( A(A^ku),A^ku)_{L^2(M,V)}\asymp\|A^ku\|_{H^1(M,V)}^2\asymp \|u\|_{H^{2k+1}(M,V)}^2,\quad  u\in H^{2k+1}(M,V), 
 \end{equation}
 see \cite{KKL}.

We shall need the following standard result in the theory of hyperbolic equations, see \cite{Las_Lions_Trig} and \cite[Thm.\ 23.2.2, Lem.\ 23.2.1]{hor_book_III} for scalar equations and
\cite[App.\ III.3]{C-B} for equations on vector bundles.

\begin{thm}
\label{lem_estimates}
Let $s\ge 0$ and $F\in L^1(\R,H^s(M,V))$ with $\supp_t(F)\subset (-\tau,+\infty)$ for some $\tau>0$.  Then the hyperbolic initial value problem
\eqref{eq_int_1} has 
a unique solution
\[
u^F\in C^1(\R,H^s(M,V))\cap C^0(\R,H^{s+1}(M,V)).
\]
Furthermore,  $u^F$ satisfies the hyperbolic estimates,  
\begin{equation}
\label{eq_estimates}
\begin{aligned}
&\|u^F(\tau_0)\|_{H^{s+1}(M,V)}\le C_{s,\tau_0}\|F\|_{L^1(\R,H^{s}(M,V))},\\
& \|\partial_t u^F(\tau_0)\|_{H^{s}(M,V)}\le C_{s,\tau_0}\|F\|_{L^1(\R,H^{s}(M,V))}, 
\end{aligned}
\end{equation}
for any $\tau_0\in \R$. 

\end{thm} 

\textbf{Remark.} 
The definition of the energy \eqref{eq_int_2} and Theorem \ref{lem_estimates} imply immediately that the function $t\mapsto E_A(F,t)$ is continuous.

Let $I_1=(t_1^-,t_1^+)$ and $I_2=(t_2^-,t_2^+)$ be bounded open  intervals in $\R$ and $\tau_0\in \R$.  In what follows we
write
\begin{align*}
&\overline{I_1}<\overline{I_2}\quad \text{if}\quad t_1^+< t_2^-,\\
&\overline{I_1}<\tau_0\quad \text{if}\quad t_1^+<\tau_0.
\end{align*}

In the course of proving Theorem \ref{thm_main}, we shall need the following result. 

\begin{lem} 
(i). Let $F_i\in L^1(\R,L^2(M,V))$ with time support bounded from below, $i=1,2$. Then for any time $\tau_0\in \R$, we have, for the both choices of the sign, 
\begin{equation}
\label{eq_energy_inner_1}
E_A(F_1\pm F_2 ,\tau_0)\asymp 
\|u^{F_1}(\tau_0)\pm u^{F_2}(\tau_0)\|_{H^{1}(M,V)}^2+ \|\partial_t u^{F_1}(\tau_0)\pm\partial_t u^{F_2}(\tau_0)\|_{L^2(M,V)}^2.
\end{equation}
(ii). Let 
$F_i\in C^\infty_0(\overline{I_i}, C^\infty(M,V))$ for some bounded open interval $I_i\subset \R$, $i=1,2$. Then 
for any $\tau_0\in \R$, $\overline{I_i}<\tau_0$, $i=1,2$,  and any $s=1,2,\dots$, 
 we have, for the both choices of the sign, 
\begin{equation}
\label{eq_energy_inner}
E_A(\partial_t^s(F_1\pm F_2) ,\tau_0)\asymp 
\|u^{F_1}(\tau_0)\pm u^{F_2}(\tau_0)\|_{H^{s+1}(M,V)}^2+ \|\partial_t u^{F_1}(\tau_0)\pm\partial_t u^{F_2}(\tau_0)\|_{H^{s}(M,V)}^2.
\end{equation}

 \end{lem}

\begin{proof} (i).  By the definition of the energy we write
\begin{align*}
E_A(F_1\pm F_2,\tau_0)=&\frac{1}{2}\|\p_t u^{F_1}(\tau_0)\pm \p_t u^{F_2}(\tau_0)\|_{L^2(M,V)}^2\\
&+\frac{1}{2}(A(u^{F_1}(\tau_0)\pm u^{F_2}(\tau_0)), u^{F_1}(\tau_0)\pm u^{F_2}(\tau_0))_{L^2(M,V)},
\end{align*}
for the both choices of the sign.  Together with \eqref{eq_3_1_1} and \eqref{eq_3_1_2} implies \eqref{eq_energy_inner_1}. 

(ii).  We have
\begin{equation}
\label{eq_3_1_3}
\begin{aligned}
E_A(\p_t^s(F_1\pm F_2),&\tau_0)=\frac{1}{2}\|\p_t \p_t^s u^{F_1}(\tau_0)\pm \p_t\p_t^s u^{F_2}(\tau_0)\|_{L^2(M,V)}^2\\
&+\frac{1}{2}(A(\p_t^s u^{F_1}(\tau_0)\pm \p_t^s u^{F_2}(\tau_0)), \p_t^s u^{F_1}(\tau_0)\pm \p_t^s u^{F_2}(\tau_0))_{L^2(M,V)}.
\end{aligned}
\end{equation}
As $\overline{I_i}<\tau_0$, we get
 \begin{equation}
 \label{eq_op}
 \partial_t^{2k} u^{F_j}(\tau_0)=(-1)^kA^ku^{F_j}(\tau_0), \quad k=1,2,\dots.
 \end{equation}

 Assume that $s$ is odd, i.e. 
 $s=2k+1$ for some $k=0,1,\dots$. The case of even $s$ can be treated in a similar way.   
 It follows from \eqref{eq_3_1_3} and 
 \eqref{eq_op} 
 that 
\begin{align*}
E_A(\p_t^s(F_1&\pm F_2),\tau_0)=\frac{1}{2}\|A^{k+1}( u^{F_1}(\tau_0)\pm  u^{F_2}(\tau_0))\|_{L^2(M,V)}^2\\
&+\frac{1}{2}(A A^k(\p_t u^{F_1}(\tau_0)\pm \p_t u^{F_2}(\tau_0)), A^k(\p_t u^{F_1}(\tau_0)\pm \p_t u^{F_2}(\tau_0)))_{L^2(M,V)}.
\end{align*}
This together with  \eqref{eq_norms_1} and  \eqref{eq_norms_2} implies that 
\begin{align*}
E_A(\p_t^s(F_1\pm F_2),\tau_0)\asymp & \| u^{F_1}(\tau_0)\pm  u^{F_2}(\tau_0)\|_{H^{2k+2}(M,V)}^2\\
&+\| \p_t u^{F_1}(\tau_0)\pm  \p_t u^{F_2}(\tau_0)\|_{H^{2k+1}(M,V)}^2,
\end{align*}
which shows \eqref{eq_energy_inner}. The proof is complete. 

 \end{proof}

The following result is needed when proving Theorem \ref{thm_main}.

\begin{lem}
\label{lem_sources_sequences} Let $s\ge 0$,  $F\in L^1(\R, H^s(M, V))$ with $\supp_t(F)\subset \overline{I_0}$, and $\tau_0\in \R$ such that $\overline{I_0}<\tau_0$. 
Then for any open interval $I_1\in\mathcal{B}_{\R}$ such that  $\overline{I_0}<\overline{I_1}<\tau_0$ and any $T\ge 0$,  there are sequences of sources
$F_j, G_j\in\mathcal{F}_{V; M,I_1}$, $j=1,2,\dots$,  such that 
\begin{equation}
\label{eq_source_-F}
(\lim_{j\to \infty} u^{F_j}(\tau_0), \lim_{j\to \infty}\partial_t u^{F_j}(\tau_0))=(u^{-\tau_TF}(\tau_0),\partial_t u^{-\tau_TF}(\tau_0))
\end{equation}
in  $H^{s+1}(M,V)\times H^s(M,V)$, and 
\begin{equation}
\label{eq_source_F}
(\lim_{j\to \infty} u^{G_j}(\tau_0), \lim_{j\to \infty}\partial_t u^{G_j}(\tau_0))=(u^{\tau_T F}(\tau_0),\partial_t u^{\tau_T F}(\tau_0)),
\end{equation}
in  $H^{s+1}(M,V)\times H^s(M,V)$
\end{lem}

\begin{proof} 
Let 
$\psi\in C^\infty(\mathbb{R},[0,1])$  be such that $\psi(t)=0$, when $t\le t_1^-$, and $\psi(t)=1$, when $t\ge t_1^+$. Assume also that $\supp(\psi')\subset I_1$. 
Then $u^{\tilde F}=-\psi u^{\tau_TF}$
is a solution to  \eqref{eq_int_1} with the right-hand side 
\[
\tilde F=-2\partial_t u^{\tau_TF}\partial_t\psi-u^{\tau_TF}\partial_t^2\psi-\psi\tau_TF=-2\partial_t u^{\tau_TF}\partial_t\psi-u^{\tau_TF}\partial_t^2\psi.
\]
We have
$\tilde F\in C^0(\R, H^s(M,V))$ and 
 $\supp_t(\tilde F)\subset I_1$.  

As the set $\mathcal{F}_{V; M,I_1}$  is dense in $C^\infty_0(\overline{I_1},C^\infty(M,V))$, there are sources $F_j\in\mathcal{F}_{M,I_1}$ such that $F_j\to \tilde F$ in $C^0(\R, H^s (M, V) )$  as $j\to\infty$.
Using the hyperbolic estimates \eqref{eq_estimates} and the fact that 
\[
(u^{\tilde F}(\tau_0),\partial_t u^{\tilde F}(\tau_0))=(u^{-\tau_TF}(\tau_0),\partial_t u^{-\tau_T F}(\tau_0)),
\]
we get \eqref{eq_source_-F}.  The existence of a sequence $G_j\in\mathcal{F}_{V; M,I_1}$ which satisfies \eqref{eq_source_F} can be seen in the same way. The proof is complete. 
\end{proof}

\section{Proof of Theorem \ref{thm_main}.}

\label{sec_main_result}

Recall that we are given a compact smooth manifold $M$ of dimension $n\ge 2$ without boundary together with its  topological and differential structures as well as a smooth positive density $d\mu$ on it. 
Furthermore, for any $N=1,2,\dots$, we are given the energy functions $\mathcal{E}^{(N)}_{V, A}$,  defined in \eqref{eq_int_4}  using the countable sets of sources $\mathcal{F}_{V; U_p,I_q}$, $p,q=1,2,\dots$.

\subsection{Information, obtained from the knowledge of $\mathcal{E}^{(N)}_{V, A}$}

\label{sec_meas_conclusion}

The purpose of this subsection is to  analyze what kind of information can be determined from the knowledge of $\mathcal{E}^{(N)}_{V, A}$.

First,  since solutions to hyperbolic equations depend continuously on data, c.f.  Theorem \ref{lem_estimates}, it is clear that the knowledge of the energy functions $\mathcal{E}^{(N)}_{V, A}$ can be extended to all $0\le T_1,\dots, T_N\in\R$ and $0\le t\in \R$,  such that the time supports of all time shifts $\tau_{T_k} F_{i_k}^{p_k,q_k}$, $k=1,\dots, N$,  of the sources $F_{i_k}^{p_k,q_k}\in \mathcal{F}_{V; U_{p_k},I_{q_k}}$ are pairwise disjoint.

Notice that in order to define the energy functions $\mathcal{E}^{(N)}_{V,A}$ in Section \ref{sec_intr_1}, we have first fixed enumerations for  $\mathcal{B}_{M}$, $\mathcal{B}_{\R}$ 
and $\mathcal{F}_{V;U_p,I_q}$. This means that in a triple $(p,q,i)\in \N\times \N\times \N$, $p$ labels  open sets $U_p$ in the basis $\mathcal{B}_{M}$ for topology on $M$,  
$q$ labels bounded open  intervals $I_q$ in the basis for topology on $\R$, and $i$ labels  the sources $F^{p,q}_{i}$ in the set $\mathcal{F}_{V;U_p,I_q}$.  
Hence, we have the following identification:
\begin{equation}
\label{eq_iden_A}
(p,q,i)\quad \longleftrightarrow\quad F^{p,q}_{i}\in \mathcal{F}_{V;U_p,I_q}.
\end{equation}
Furthermore, for any $p\in \N$ and $q\in \N$,  when writing  $U_p$ and $I_q$, we mean that $U_p\in \mathcal{B}_{M}$ and $I_q\in \mathcal{B}_\R$.

Let $X=\{(i_j)_{j=1}^{\infty}:i_j\in\N\}$
 be the set of all sequences of integers. 
Let $(p_0,q_0,i_0)\in \N\times \N\times \N$ and let $\tau_0\in\R$ be such that $\overline{I_{q_0}}<\tau_0$.  Assume that $T\ge 0$ and $q_1\in\N$ are  such that $\overline{I_{q_0}}<\overline{I_{q_1}}<\tau_0$. Consider the set
\[
X_{(p_0,q_0,i_0),q_1,T,\tau_0}=\{(i_j)_{j=1}^\infty\in X: \lim_{j\to \infty} E_A(F^{1,q_1}_{i_j}+\tau_T F^{p_0,q_0}_{i_0},\tau_0)=0\},
\]
where $F^{1,q_1}_{i_j}\in \mathcal{F}_{V;M,I_{q_1}}$ and $F^{p_0,q_0}_{i_0}\in \mathcal{F}_{V;U_{p_0},I_{q_0}}$.

First we claim that the set $X_{(p_0,q_0,i_0),q_1,T,\tau_0}\ne\emptyset$. This follows immediately from Lemma \ref{lem_sources_sequences}
with the help of  \eqref{eq_energy_inner_1}.

\begin{lem}
\label{lem_sequences_uniqueness}

Given the energy functions $\mathcal{E}^{(N)}_{V,A}$, $N=1,2,\dots$, we can determine the set $X_{(p_0,q_0,i_0),q_1,T,\tau_0}$. \footnote{Here and in what follows when we say that given the energy functions $\mathcal{E}^{(N)}_{V,A}$,  we can determine a set or a function,  we mean a uniqueness statement: if $\mathcal{E}^{(N)}_{V_1,A_1}=\mathcal{E}^{(N)}_{V_2,A_2}$ then the sets or functions in question, corresponding to the different vector bundles,  are equal.}

\end{lem}

\begin{proof}

The knowledge of  the functions $\mathcal{E}^{(N)}_{V,A}$ implies the knowledge of the values $E_A(F^{1,q_1}_{i_j}+\tau_T F^{p_0,q_0}_{i_0},\tau_0)$ for any index $i_j$, since $\supp_t(\tau_T F^{p_0,q_0}_{i_0})\cap\supp_t(F^{1,q_1}_{i_j})=\emptyset$, $T\ge 0$. Thus, for any sequence $(i_j)_{j=1}^\infty\in X$, we can determine whether it belongs to the set $X_{(p_0,q_0,i_0),q_1,T,\tau_0}$.  In this sense, we can construct the set $X_{(p_0,q_0,i_0),q_1,T,\tau_0}$.

\end{proof}

Recall that the energy functions $\mathcal{E}^{(N)}_{V, A}$ encode only information about the energy, which is produced by the sources of the form 
$F= \sum_{k=1}^N \tau_{T_k}F^{p_k,q_k}_{i_k},\quad F^{p_k,q_k}_{i_k}\in \mathcal{F}_{V;U_{p_k},I_{q_k}}$.  However, in order to recover the vector bundle $V$, we need to know values of the energy for an arbitrary linear combination of the form
$\sum_{k=1}^N a_{k} \tau_{T_k}F^{p_k,q_k}_{i_k}$, $a_k\in \R$, evaluated at least at a suitable time $\tau_0$. This is our next goal. 

\begin{lem}
\label{lem_energy_F_1-F_2}  

Given the energy functions $\mathcal{E}^{(N)}_{V,A}$, $N=1,2,\dots$, for any $(p_1,q_1,i_1)\in \N\times\N\times\N$, $(p_2,q_2,i_2)\in \N\times \N\times\N$, $T\ge 0$ and $\tau_0\in \R$ such that $\overline{I_{q_k}}<\tau_0$, $k=1,2$, we can determine the functions
\[
((p_1,q_1,i_1), (p_2,q_2,i_2),T;\tau_0)\mapsto E_A(F_{i_1}^{p_1,q_1}-\tau_T F_{i_2}^{p_2,q_2},\tau_0).
\]

\end{lem}

\begin{proof}
Let $q_0\in \N$ be such that $\overline{I_{q_k}}<\overline{I_{q_0}}<\tau_0$, $k=1,2$.  Then by Lemma \ref{lem_sequences_uniqueness}, using the energy functions $\mathcal{E}^{(N)}_{V,A}$, $N=1,2,\dots$, we can find a sequence $(i_j)_{j=1}^\infty\in X$ such that 
\[
\lim_{j\to\infty} E_A(F_{i_j}^{1,q_0}+\tau_T F^{p_2,q_2}_{i_2},\tau_0)=0. 
\]
Thus,
\begin{equation}
\label{eq_4_1}
E_A(F^{p_1,q_1}_{i_1}-\tau_T F^{p_2,q_2}_{i_2},\tau_0)=\lim_{j\to \infty} E_A(F^{p_1,q_1}_{i_1}+F_{i_j}^{1,q_0}).
\end{equation}
As the time supports of $F^{p_1,q_1}_{i_1}$ and $F_{i_j}^{1,q_0}$ are disjoint, the  energy in \eqref{eq_4_1} can be computed using $\mathcal{E}^{(N)}_{V,A}$. The proof is complete. 
\end{proof}

Let $(p_0,q_0,i_0)\in \N\times \N\times \N$ and $\tau_0\in\R$ be such that $\overline{I_{q_0}}<\tau_0$.  Let $T\ge 0$ and $q_1\in\N$ be such that $\overline{I_{q_0}}<\overline{I_{q_1}}<\tau_0$. Consider the set
\[
Y_{(p_0,q_0,i_0),q_1,T,\tau_0}=\{(i_j)_{j=1}^\infty\in X: \lim_{j\to \infty} E_A(F^{1,q_1}_{i_j}-\tau_T F^{p_0,q_0}_{i_0},\tau_0)=0\},
\]
where $F^{1,q_1}_{i_j}\in \mathcal{F}_{V;M,I_{q_1}}$ and $F^{p_0,q_0}_{i_0}\in \mathcal{F}_{V;U_{p_0},I_{q_0}}$. 
It  follows immediately from \eqref{eq_source_F} in  Lemma \ref{lem_sources_sequences}
 that  $Y_{(p_0,q_0,i_0),q_1,T,\tau_0}\ne\emptyset$.

\begin{lem}

\label{lem_sequences_uniqueness_2}
Given the energy functions $\mathcal{E}^{(N)}_{V,A}$, $N=1,2,\dots$, we can determine the set $Y_{(p_0,q_0,i_0),q_1,T,\tau_0}$.

\end{lem}

\begin{proof}

By Lemma \ref{lem_energy_F_1-F_2}, the knowledge of  the functions $\mathcal{E}^{(N)}_{V,A}$ implies the knowledge of the values $E_A(F^{1,q_1}_{i_j}-\tau_T F^{p_0,q_0}_{i_0},\tau_0)$ for any index $i_j$. 
Hence, for any sequence $(i_j)_{j=1}^\infty\in X$, we can determine whether  it belongs to the set $Y_{(p_0,q_0,i_0),q_1,T,\tau_0}$, and therefore, we can determine the set $Y_{(p_0,q_0,i_0),q_1,T,\tau_0}$.

\end{proof}

\begin{lem}
\label{cor_possible_sums_of_energy}
Given the energy functions $\mathcal{E}^{(N)}_{V,A}$, $N=1,2,\dots$, for any $N=1,2,\dots$, $(p_k,q_k,i_k)\in \N\times\N\times\N$, $k=1,\dots, N$, $0\le T_1,\dots, T_N\in \R$, $a_1,\dots,a_N\in \Q$, and $\tau_0\in\R$ such that $\overline{I_{q_k}}<\tau_0$, $k=1,\dots,N$, we can determine the functions
\begin{align*}
\bigg((p_1,q_1,i_1),\dots, (p_N,q_N,i_N),T_1,\dots, T_N, a_1,\dots,a_N; \tau_0 \bigg)\\
\mapsto E_A\bigg( \sum_{k=1}^N a_k\tau_{T_k}F^{p_k,q_k}_{i_k}; \tau_0\bigg). 
\end{align*}

\end{lem}

\begin{proof}
Let $a_k=m_k/n_k$, $m_k\in \Z$, $n_k\in \N$, $k=1,\dots,N$, and $l=\Pi_{k=1}^N n_k$.  Then we have
\[
E_A\bigg( \sum_{k=1}^N a_k\tau_{T_k}F^{p_k,q_k}_{i_k}; \tau_0\bigg)=l^{-2}E_A\bigg( \sum_{k=1}^N \tilde m_k\tau_{T_k}F^{p_k,q_k}_{i_k}; \tau_0\bigg),
\]
where $\tilde m_k=a_kl\in \Z$. We write
\[
E_A\bigg( \sum_{k=1}^N \tilde m_k\tau_{T_k}F^{p_k,q_k}_{i_k}; \tau_0\bigg)=E_A\bigg( \sum_{k=1}^N \sum_{s=1}^{|\tilde m_k|} \textrm{sign}(\tilde m_k)\tau_{T_k}F^{p_k,q_k}_{i_k}; \tau_0\bigg).
\]
Let $r_{ks}\in \N$,  $s=1,\dots,|\tilde m_k|$, $k=1,\dots,N$,  be such that for all $l=1,\dots,N$,
\begin{align*}
\overline{I_{q_l}}<\overline{I_{r_{11}}}<\dots <\overline{I_{r_{1|\tilde m_1|}}}< \overline{I_{r_{21}}}<\dots <\overline{I_{r_{2|\tilde m_2|}}}<\dots <\overline{I_{r_{N1}}}<\dots <\overline{I_{r_{N|\tilde m_N|}}}<\tau_0.
\end{align*}
Then by Lemmas \ref{lem_sequences_uniqueness} and \ref{lem_sequences_uniqueness_2}, the knowledge of the energy functions $\mathcal{E}^{(N)}_{V,A}$ allows us to find sequences $(i_{ks,j})_{j=1}^\infty\in X$ such that 
for $s=1,\dots,|\tilde m_k|$, $k=1,\dots,N$, 
\[
\lim_{j\to \infty} E_A(F^{1,r_{ks}}_{i_{ks,j}}  -\textrm{sign}(\tilde m_k)\tau_{T_k}F^{p_k,q_k}_{i_k};\tau_0)=0.
\]
Here $F^{1,r_{ks}}_{i_{ks,j}}\in \mathcal{F}_{V; M, I_{r_{ks}}}$. 
Hence,
\[
E_A\bigg( \sum_{k=1}^N a_k\tau_{T_k}F^{p_k,q_k}_{i_k}; \tau_0\bigg)=l^{-2}\lim_{j\to\infty} E_A\bigg( \sum_{k=1}^N \sum_{s=1}^{|\tilde m_k|} F^{1,r_{ks}}_{i_{ks,j}}; \tau_0\bigg),
\]
which can be computed as the time supports of all sources $F^{1,r_{ks}}_{i_{ks,j}}$ are pairwise disjoint. 
The proof is complete.

\end{proof}

Since solutions to hyperbolic equations depend continuously on data, we get the following corollary. 

\begin{cor}
\label{cor_energy_h}
Given the energy functions $\mathcal{E}^{(N)}_{V,A}$, $N=1,2,\dots$, for any $N=1,2,\dots$, $(p_k,q_k,i_k)\in \N\times\N\times\N$, $k=1,\dots, N$, $0\le T_1,\dots, T_N\in \R$, $a_1,\dots,a_N\in \R$, and $\tau_0\in\R$ such that $\overline{I_{q_k}}<\tau_0$, $k=1,\dots,N$, we can determine the functions
\begin{align*}
\bigg((p_1,q_1,i_1),\dots, (p_N,q_N,i_N),T_1,\dots, T_N, a_1,\dots,a_N; \tau_0 \bigg)\\
\mapsto E_A\bigg( \sum_{k=1}^N a_k\tau_{T_k}F^{p_k,q_k}_{i_k}; \tau_0\bigg). 
\end{align*}

\end{cor}

In order to reconstruct the vector bundle $V$, we have to be able to compute the energy of  time derivatives of the sources from $\mathcal{F}_{V;M,I_q}$. We now proceed to do so. 

\begin{lem}
\label{lem_deriv}
Given the energy functions $\mathcal{E}^{(N)}_{V,A}$, $N=1,2,\dots$, for any $q_0\in \N$, $i_0\in\N$ and $\tau_0\in\R$ such that $\overline{I_{q_0}}<\tau_0$, and any $s=1,2,\dots$, we can determine the functions
\[
(q_0,i_0,s;\tau_0)\mapsto E_A(\p_t^s F^{1,q_0}_{i_0};\tau_0). 
\]
\end{lem}

\begin{proof} 

Using the forward difference approximation to the $s$-th order derivative and  Corollary \ref{cor_energy_h},  we can compute
\[
E_A(\p_t^s F^{1,q_0}_{i_0};\tau_0)=\lim_{h\to 0, h>0}\frac{1}{h^{2s}}E_A\bigg(\sum_{k=0}^s(-1)^k
\begin{pmatrix} s\\
k
\end{pmatrix} \tau_{(s-k)h}F^{1,q_0}_{i_0};\tau_0\bigg)
\]
where $\begin{pmatrix} s\\
k
\end{pmatrix}$ are the binomial coefficients. The proof is complete.

 \end{proof}

Similarly, arguing as in the proof of Lemma \ref{lem_deriv},  we have the following result. 

\begin{lem}
\label{lem_sec_3_1}
Given the energy functions $\mathcal{E}^{(N)}_{V,A}$, $N=1,2,\dots$, for any $q_1,q_2\in \N$, $i_1,i_2\in\N$ and $\tau_0\in\R$ such that $\overline{I_{q_k}}<\tau_0$, $k=1,2$, and any $s=1,2,\dots$, we can determine the functions
\begin{align*}
((q_1,i_1),(q_2,i_2), s;\tau_0)&\mapsto E_A(\p_t^s (F^{1,q_1}_{i_1}- F^{1,q_2}_{i_2});\tau_0),\\
((q_1,i_1),(q_2,i_2), s;\tau_0)&\mapsto E_A(\p_t^s F^{1,q_1}_{i_1}- F^{1,q_2}_{i_2};\tau_0).
\end{align*}
\end{lem}

Our next goal is to show that the knowledge of the energy functions $\mathcal{E}^{(N)}_{V,A}$ allows us to compute $L^2$ -- inner products between sources from the set $\mathcal{F}_{V; U_p,I_q}$ and the time derivatives of waves corresponding to such sources. To that end we shall need the following observation. 

\begin{lem}
\label{lem_energy_F_1-F_2_2}  

Given the energy functions $\mathcal{E}^{(N)}_{V,A}$, $N=1,2,\dots$, for any $(p_1,q_1,i_1)\in \N\times\N\times\N$, $(p_2,q_2,i_2)\in \N\times \N\times\N$ such that $\overline{I_{q_1}}<\overline{I_{q_2}}$,  $T\ge 0$ and any $\tau_0\in \R$ such that $\tau_0\in I_{q_2}$, we can determine the functions
\[
((p_1,q_1,i_1), (p_2,q_2,i_2),T;\tau_0)\mapsto E_A(\tau_T F_{i_1}^{p_1,q_1}- F_{i_2}^{p_2,q_2},\tau_0).
\]

\end{lem}

\begin{proof}
Let $q_0\in \N$ be such that $\overline{I_{q_1}}<\overline{I_{q_0}}<\overline{I_{q_2}}$.  Then by Lemma \ref{lem_sequences_uniqueness}, using the energy functions $\mathcal{E}^{(N)}_{V,A}$, $N=1,2,\dots$, we can find a sequences $(i_j)_{j=1}^\infty\in X$ such that 
\[
\lim_{j\to\infty} E_A(F_{i_j}^{1,q_0}+\tau_T F^{p_1,q_1}_{i_1},\tau_0)=0. 
\]
Hence,
\[
E_A(\tau_T F^{p_1,q_1}_{i_1}- F^{p_2,q_2}_{i_2},\tau_0)=\lim_{j\to \infty} E_A(-F_{i_j}^{1,q_0}-F^{p_2,q_2}_{i_2},\tau_0)=\lim_{j\to \infty} E_A(F_{i_j}^{1,q_0}+F^{p_2,q_2}_{i_2},\tau_0),
\]
which can be computed using $\mathcal{E}^{(N)}_{V,A}$, as the time supports of the sources $F^{p_2,q_2}_{i_2}$ and $F_{i_j}^{1,q_0}$ are disjoint. This completes the proof. 
\end{proof}

In what follows, whenever convenient, we shall use the notation,
\[
u(t;F)=u^F(t). 
\]
We have the following result. 
\begin{lem}
\label{lem_measurements_derivative}

Given the energy functions $\mathcal{E}^{(N)}_{V,A}$, $N=1,2,\dots$, for any $(q_1,i_1)\in \N\times\N$, $(p_2,q_2,i_2)\in \N\times \N\times\N$ such that $\overline{I_{q_1}}<\overline{I_{q_2}}$,  and any $\tau_0\in \R$ such that $\tau_0\in I_{q_2}$, we can determine the functions
\[
((q_1,i_1),(p_2,q_2,i_2); \tau_0)\mapsto (F^{p_2,q_2}_{i_2}(\tau_0), \partial_t u(\tau_0; F^{1,q_1}_{i_1}))_{L^2(M,V)}.
\]

\end{lem}

\begin{proof} Let $F\in L^1(\R,L^2(M,V))$. Then since $A$ is self-adjoint operator, we get, taking the time derivative in the sense of distributions, 
\begin{equation}
\label{eq_conservation_energy}
\begin{aligned}
\partial_t E_A(F,t)&=\int_M(\langle \partial^2_t u^F(x,t),\partial_t u^F(x,t)\rangle_x+
\langle A u^F(x,t),\partial_t u^F(x,t)\rangle_x)d\mu(x)\\
&=\int_M\langle A u^F(x,t)+\partial^2_t u^F(x,t),\partial_t u^F(x,t)\rangle_xd\mu(x)\\
&=\int_M\langle F(x,t),\partial_t u^F(x,t)\rangle_xd\mu(x).
\end{aligned}
\end{equation}
Thus, we obtain that
\begin{align*}
&\partial_t[E_A(F^{1,q_1}_{i_1}+F^{p_2,q_2}_{i_2},t)-E_A(F^{1,q_1}_{i_1}-F^{p_2,q_2}_{i_2},t)]\\
&=2\int_M(\langle F^{1,q_1}_{i_1}(t),\partial_t u(t; F^{p_2,q_2}_{i_2})\rangle_x+\langle F^{p_2,q_2}_{i_2}(t),\partial_t u(t; F^{1,q_1}_{i_1})\rangle_x)d\mu(x).
\end{align*}
As $\overline{I_{q_1}}<\overline{I_{q_2}}$, we have 
\[
\supp(u(\cdot; F^{p_2,q_2}_{i_2}))\cap\supp(F^{1,q_1}_{i_1})=\emptyset.
\]
Hence,
\[
\partial_t[E_A(F^{1,q_1}_{i_1}+F^{p_2,q_2}_{i_2},t)-E_A(F^{1,q_1}_{i_1}-F^{p_2,q_2}_{i_2},t)]=
2 (F^{p_2,q_2}_{i_2}(t),\partial_t u(t; F^{1,q_1}_{i_1}))_{L^2(M,V)}.
\]
As for any $t\in I_{q_2}$, the energy $E_A(F^{1,q_1}_{i_1}+F^{p_2,q_2}_{i_2},t)-E_A(F^{1,q_1}_{i_1}-F^{p_2,q_2}_{i_2},t)$ can be computed using energy functions $\mathcal{E}^{(N)}_{V,A}$ and Lemma \ref{lem_energy_F_1-F_2_2},
 the claim follows. The proof is complete. 

\end{proof}

\begin{cor}
\label{cor_measurements_derivative}
Given the energy functions $\mathcal{E}^{(N)}_{V,A}$, $N=1,2,\dots$,  when $(q_1,i_1)\in \N\times\N$, $(p_2,q_2,i_2)\in \N\times \N\times\N$ such that $\overline{I_{q_1}}<\overline{I_{q_2}}$,  and any $\tau_0,T_0\in \R$ such that $\tau_0\in I_{q_2}$ and $\overline{I_{q_2}}<T_0$, we can determine the functions
\[
((q_1,i_1),(p_2,q_2,i_2); \tau_0,T_0)\mapsto (F^{p_2,q_2}_{i_2}(\tau_0), \partial_t u(T_0; F^{1,q_1}_{i_1}))_{L^2(M,V)}.
\]

\end{cor}

\begin{proof}
The claim follows by applying  Lemma \ref{lem_measurements_derivative} with $\tau_{T_0-\tau_0}F^{1,q_1}_{i_1}$ instead of $F^{1,q_1}_{i_1}$ and the general  fact that
$u^{\tau_{T}F}(t)=\tau_{T}u^{F}(t)=u^{F}(t+T)$. 
\end{proof}

The following consequence of Corollary \ref{cor_measurements_derivative} will be useful when recovering the operator $A$. 
\begin{cor}
\label{cor_measurements_deriv_s}

Given the energy functions $\mathcal{E}^{(N)}_{V,A}$, $N=1,2,\dots$,  for any $s=1,2,\dots$,  $(q_1,i_1)\in \N\times\N$, $(p_2,q_2,i_2)\in \N\times \N\times\N$ such that $\overline{I_{q_1}}<\overline{I_{q_2}}$,  and any $\tau_0,T_0\in \R$ such that $\tau_0\in I_{q_2}$ and $\overline{I_{q_2}}<T_0$, we can determine the functions
\[
((q_1,i_1),(p_2,q_2,i_2), s; \tau_0,T_0)\mapsto (F^{p_2,q_2}_{i_2}(\tau_0), \partial_t^s u(T_0; F^{1,q_1}_{i_1}))_{L^2(M,V)}.
\]

\end{cor}

When recovering the inner product in the fibers $\pi^{-1}(x)$ of the vector bundle $V$, we shall need the following observation. 

\begin{cor}
\label{cor_inner_f}

Given the energy functions $\mathcal{E}^{(N)}_{V,A}$, $N=1,2,\dots$,  for any $s=1,2,\dots$,  $(q_1,i_1)\in \N\times\N$, $(p_2,q_2,i_2)\in \N\times \N\times\N$ such that $\overline{I_{q_1}}<\overline{I_{q_2}}$,  and any $\tau_0,T_0\in \R$ such that $\tau_0\in I_{q_2}$ and $\overline{I_{q_2}}<T_0$, we can determine the functions
\[
((q_1,i_1),(p_2,q_2,i_2), s; \tau_0,T_0)\mapsto (F^{p_2,q_2}_{i_2}(\tau_0), A u(T_0; F^{1,q_1}_{i_1}))_{L^2(M,V)}.
\]

\end{cor}

Corollary \ref{cor_inner_f} follows  from Corollary \ref{cor_measurements_deriv_s} and  the fact that $F^{1,q_1}_{i_1}(T_0)=0$, and therefore, $A u(T_0; F^{1,q_1}_{i_1})=-\p_t^2 u(T_0; F^{1,q_1}_{i_1})$.

\subsection{Generalized sources} 

The purpose of this subsection is to introduce the space of generalized sources, which will allow us to view Sobolev spaces $H^s(M,V)$ as sets of waves, corresponding to such generalized sources. 
We note that the construction of the space of generalized sources is well known in control theory for PDE, see \cite{Las_Trig_2000, KL-Dirac,  Russell_1978}.

Let $q_0\in\N$. Then   $I_{q_0}\in \mathcal{B}_{\R}$ and we recall that 
\[
\mathcal{F}_{V; M,I_{q_0}}=\{F^{1,q_0}_i: i=1,2,\dots\}\subset C_0^\infty(\overline{I_{q_0}}, C^\infty(M,V)), 
\] 
where the inclusion is dense. 
Denote as before by
\[
X=\{(i_j)_{j=1}^\infty:i_j\in \N\}
\]
the set of all sequences of integers. 

For any $s=0,1,2,\dots$ and $\tau_0\in \R$ such that $\overline{I_{q_0}}<\tau_0$,  let us introduce the following space of sequences of sources,
  \[
 \mathcal{F}^s_{q_0,\tau_0}=\{(F^{1,q_0}_{i_j})_{j=1}^\infty
 : F^{1,q_0}_{i_j}\in \mathcal{F}_{V; M,I_{q_0}},\ \lim_{j,k\to\infty} E_A(\partial_t^s(F^{1,q_0}_{i_j}-F^{1,q_0}_{i_k});\tau_0)=0\},
 \]
 and the  corresponding set of indices,
 \[
 Z^s_{q_0,\tau_0}=\{(i_j)_{j=1}^\infty\in X: (F^{1,q_0}_{i_j})_{j=1}^\infty\in \mathcal{F}^s_{q_0,\tau_0} \}. 
 \]

 \begin{rem}
Given the energy functions $\mathcal{E}^{(N)}_{V,A}$, $N=1,2,\dots$, the sets $Z^s_{q_0,\tau_0}$ can be constructed. Indeed, by Lemma \ref{lem_sec_3_1}, for any sequence of indices $(i_j)_{j=1}^\infty\in X$,  we can compute the energy $E_A(\partial_t^s(F^{1,q_0}_{i_j}-F^{1,q_0}_{i_k});\tau_0)$ and thus, we can check whether the given sequence of indices $(i_j)_{j=1}^\infty$ belongs to $Z^s_{q_0,\tau_0}$. 

\end{rem}

It follows from  \eqref{eq_energy_inner} that for any  $(F^{1,q_0}_{i_j})_{j=1}^\infty\in \mathcal{F}^s_{q_0,\tau_0}$,  the sequences $(u(\tau_0; F^{1,q_0}_{i_j}))_{j=1}^\infty$ and  $(\partial_t u(\tau_0; F^{1,q_0}_{i_j}))_{j=1}^\infty$ are  Cauchy sequences in $H^{s+1}(M,V)$ and $H^{s}(M,V)$, respectively. Since the space $ H^{s}(M,V)$ is complete,  we can define the operator
\begin{align*}
W:\mathcal{F}^s_{q_0,\tau_0}&\to H^{s+1}(M,V)\times H^{s}(M,V), \\
W((F^{1,q_0}_{i_j})_{j=1}^\infty)&=\bigg(\lim_{j\to\infty} u(\tau_0; F^{1,q_0}_{i_j}),\lim_{j\to\infty} \partial_t u(\tau_0; F^{1,q_0}_{i_j})\bigg).
\end{align*}
We also  define a semi-norm on the space $\mathcal{F}^s_{q_0,\tau_0}$ of sequences of  sources, given by
\begin{equation}
\label{eq_semi-norm}
\|(F^{1,q_0}_{i_j})_{j=1}^\infty\|_{\mathcal{F}^s_{q_0,\tau_0}}=\lim_{j\to \infty}\sqrt{ E_A(\partial_t^sF^{1,q_0}_{i_j},\tau_0)}.
\end{equation}

\begin{rem}
Given the energy functions $\mathcal{E}^{(N)}_{V,A}$, $N=1,2,\dots$, Lemma \ref{lem_deriv} allows us to compute the semi-norm \eqref{eq_semi-norm} for any sequence of indices $(i_j)_{j=1}^\infty\in Z^s_{q_0,\tau_0}$.

\end{rem}

We say that $(F^{1,q_0}_{i^1_j})_{j=1}^\infty\in \mathcal{F}^s_{q_0,\tau_0}$ is equivalent to $(F^{1,q_0}_{i^2_j})_{j=1}^\infty\in \mathcal{F}^s_{q_0,\tau_0}$ and write
$(F^{1,q_0}_{i^1_j})_{j=1}^\infty\sim (F^{1,q_0}_{i^2_j})_{j=1}^\infty$
if
\begin{align*}
  & \lim_{j\to\infty} u(\tau_0; F^{1,q_0}_{i^1_j})=\lim_{j\to\infty} u(\tau_0; F^{1,q_0}_{i^2_j}) \quad\text{in}\quad H^{s+1}(M,V),\\
  & \lim_{j\to\infty} \partial_t u(\tau_0; F^{1,q_0}_{i^1_j})=\lim_{j\to\infty} \partial_t u(\tau_0; F^{1,q_0}_{i^2_j})\quad \text{in} \quad H^{s}(M,V).
  \end{align*}
This is equivalent to that fact that $\|(F^{1,q_0}_{i^1_j}- F^{1,q_0}_{i^2_j})_{j=1}^\infty\|_{\mathcal{F}^s_{q_0,\tau_0}}=0$.
Further, we define the space $\mathcal{F}^s_{q_0,\tau_0}/\sim$ and \eqref{eq_semi-norm} becomes a norm on this space.  The operator $W$, extended to  the linear  space $\mathcal{F}^s_{q_0,\tau_0}/\sim$,  is then continuous. 
Elements of the space $\mathcal{F}^s_{q_0,\tau_0}/\sim$ will be denoted by $[(F^{1,q_0}_{i_j})_{j=1}^\infty]$, $(i_j)_{j=1}^\infty\in  Z^s_{q_0,\tau_0}$.

Finally, we complete $\mathcal{F}^s_{q_0,\tau_0}/\sim$ with respect to the norm \eqref{eq_semi-norm} and denote this completion by  $\overline{\mathcal{F}}^s_{q_0,\tau_0}$.
The space $\overline{\mathcal{F}}^s_{q_0,\tau_0}$ is called the space of \emph{generalized sources}. A typical element of $\overline{\mathcal{F}}^s_{q_0,\tau_0}$ is the set of equivalence classes of  Cauchy  sequences $([(F^{1,q_0}_{i^k_j})_{j=1}^\infty])_{k=1}^\infty$, $(i_j^k)_{j=1}^\infty\in Z^s_{q_0,\tau_0}$ for all $k=1,2,\dots$, i.e., 
\[
\|[(F^{1,q_0}_{i^k_j})_{j=1}^\infty]-[(F^{1,q_0}_{i^l_j})_{j=1}^\infty]\|_{\mathcal{F}^s_{q_0,\tau_0}}\to 0,\quad \textrm{as}\ k,l\to \infty.
\]

\begin{rem}
Given the energy functions $\mathcal{E}^{(N)}_{V,A}$, $N=1,2,\dots$, we can construct the space of the "generalized indices", which corresponds to the space of generalized sources $\overline{\mathcal{F}}^s_{q_0,\tau_0}$, in the identification \eqref{eq_iden_A}.  
For simplicity of notation, we shall refrain from introducing this space explicitly. 

\end{rem}

For any $\hat F\in \overline{\mathcal{F}}^s_{q_0,\tau_0}$, $\hat F=[([(F^{1,q_0}_{i^k_j})_{j=1}^\infty])_{k=1}^\infty]$, $F^{1,q_0}_{i^k_j}\in \mathcal{F}_{V;M,I_{q_0}}$,  we define
\begin{equation}
\label{eq_sec_5_1}
\begin{aligned}
u^{\hat{F}}(\tau_0)=\lim_{k\to\infty}\lim_{j\to\infty} u(\tau_0; F^{1,q_0}_{i^k_j}),\quad
\partial_tu^{\hat{F}}(\tau_0)=\lim_{k\to\infty}\lim_{j\to\infty} \partial_t u(\tau_0; F^{1,q_0}_{i^k_j}).
\end{aligned}
\end{equation}
 We extend continuously  the wave operator $W$ to the space of generalized sources by 
\begin{equation}
\label{eq_sec_4_2}
\begin{aligned}
W: \overline{\mathcal{F}}^s_{q_0,\tau_0}&\to H^{s+1}(M,V)\times H^{s}(M,V),\\
W(\hat{F})&=(u^{\hat{F}}(\tau_0),\partial_tu^{\hat{F}}(\tau_0)).
\end{aligned}
\end{equation}

\begin{lem} 
\label{lem_wave_op-new}

The operator $W$, given by \eqref{eq_sec_4_2}, is bijective. 
\end{lem}

\begin{proof}

The injectivity of $W$ follows from the definition of the space $\overline{\mathcal{F}}^s_{q_0,\tau_0}$. 

Let us prove that $W$ is surjective. 
First we show that for any $(a,b)\in C^\infty(M,V)\times C^\infty(M,V)$, there is
 $F\in C^\infty_0(\overline{I_{q_0}}, C^\infty (M, V))$  such that $u^F(\tau_0)=a$ and $\partial_t u^F(\tau_0)=b$,  where $u^F$ is a solution to \eqref{eq_int_1}. Indeed, let $v$ be a solution to
\begin{align*}
&(\partial_t^2+A)v =0, \quad \text{in} \ M\times \mathbb{R}\\
&v|_{t=\tau_0}=a,\quad \partial_t v|_{t=\tau_0}=b.
\end{align*}
Writing $I_{q_0}=(t_0^-,t_0^+)$, we 
 let $\psi\in C^\infty(\mathbb{R},[0,1])$ be such that  $\psi(t)=0$ if $t\le t_0^-$ and $\psi(t)=1$ if $t\ge t_0^+$. Then $u^F(x,t)=v(x,t)\psi(t)$ is a solution to \eqref{eq_int_1} with  $F=2\partial_t v\partial_t \psi + v\partial_t^2\psi\in 
 C^\infty_0(\overline{I_{q_0}},C^\infty(M, V) )$. 
 Furthermore, $u^F(\tau_0)=a$ and $\partial_t u^F(\tau_0)=b$.

 Now let $(a,b)\in H^{s+1}(M,V)\times H^{s}(M,V)$. Since $C^\infty(M,V)$ is dense in  all Sobolev spaces, there are sequences $(a_k)_{k=1}^\infty$, $(b_k)_{k=1}^\infty$ with   $a_k,b_k\in C^\infty(M,V)$ such that
 \begin{equation}
 \label{eq_sec_4_3}
  a=\lim_{k\to\infty}a_k\quad \text{in}\ H^{s+1}(M,V),\quad b=\lim_{k\to\infty}b_k\quad \text{in}\ H^{s}(M,V).
  \end{equation}
 By the first part of the proof, there are sources $F_k\in C^\infty_0(\overline{I_{q_0}}, C^\infty (M, V) )$   such that $u^{F_k}(\tau_0)=a_k$ and $\partial_t u^{F_k}(\tau_0)=b_k$. Since $\mathcal{F}_{V; M,I_{q_0}}$ is dense in  $C^\infty_0(\overline{I_{q_0}},C^\infty(M, V))$, for any $F_k$ there is a sequence of sources $(F^{1,q_0}_{i^k_j})_{j=1}^\infty$,   $F^{1,q_0}_{i^k_j}\in\mathcal{F}_{V; M,I_{q_0}}$,  such that $\lim_{j\to\infty}F^{1,q_0}_{i^k_j}=F_k$ in $C^\infty_0(\overline{I_{q_0}},C^\infty(M, V))$.  
  Thus, \eqref{eq_estimates} yields that 
 \begin{equation}
 \label{eq_sec_4_4}
 \begin{aligned}
  &\lim_{j\to\infty} u(\tau_0; F^{1,q_0}_{i^k_j})=a_k\quad \textrm{in}Ê\quad  H^{s+1}(M,V), \\ 
 &\lim_{j\to\infty}\partial_t u(\tau_0; F^{1,q_0}_{i^k_j})=b_k\quad\textrm{in}\quad  H^s(M,V).
 \end{aligned}
 \end{equation} 
 We conclude from \eqref{eq_energy_inner}  and \eqref{eq_sec_4_4} that for each $k$, 
\[
 \lim_{j,l\to \infty} E_A(\p_t^s(F^{1,q_0}_{i^k_j}-F^{1,q_0}_{i^k_l});\tau_0)=0.
\] 
 Hence, for each $k$, $(i^k_j)_{j=1}^\infty\in Z^s_{q_0,\tau_0}$.
 
 Furthermore, 
 it follows from \eqref{eq_sec_4_3} and \eqref{eq_sec_4_4} that 
 \begin{align*}
  &\lim_{k\to \infty}\lim_{j\to\infty} u(\tau_0; F^{1,q_0}_{i^k_j})=a\quad \textrm{in}Ê\quad  H^{s+1}(M,V), \\ 
 &\lim_{k\to \infty}\lim_{j\to\infty}\partial_t u(\tau_0; F^{1,q_0}_{i^k_j})=b\quad\textrm{in}\quad  H^s(M,V).
 \end{align*}
 This together with  \eqref{eq_energy_inner}  implies that 
 \[
 \lim_{j\to\infty} E_A(\p_t^s(F^{1,q_0}_{i^k_j}-F^{1,q_0}_{i^l_j});\tau_0)\to 0,\quad\textrm{as}\quad k,l\to \infty,
 \]
 and therefore, 
  $[([(F^{1,q_0}_{i^k_j})_{j=1}^\infty])_{k=1}^\infty]\in \overline{\mathcal{F}}^s_{q_0,\tau_0}$. The proof is complete.

\end{proof}

We also define the space
\[
\overline{\mathcal{F}}^\infty_{q_0,\tau_0}=\bigcap_{s=0,1,2,\dots}\overline{\mathcal{F}}^s_{q_0,\tau_0}.
\]
It is clear that the map 
\begin{equation}
\label{eq_W-infty}
W^\infty:\overline{\mathcal{F}}^\infty_{q_0,\tau_0}\to C^\infty(M,V)\times C^\infty(M,V), 
\end{equation}
induced by $W$, 
 is bijective.

\subsection{Reconstruction of the vector bundle and the Riemannian structure on it}

In what follows we shall need to work with the space of distributions with values in the vector bundle $V$, 
\[
\mathcal{D}'(M,V)=(C^\infty(M,V\otimes \Omega))',
\]
where $\Omega$ is the density bundle over $M$, see \cite[Chapter 18]{hor_book_III}. 

In particular, we have the delta distribution $\delta_{x_0}\in \mathcal{D}'(M,\R)=(C^\infty(M,\Omega))'$ at $x_0\in M$, given by 
\[
\int_M \delta_{x_0}(x)\phi(x) d\mu(x)=  \delta_{x_0}(\phi d\mu)=\phi(x_0),\quad \phi\in C^\infty(M, \R).
\]

Furthermore, for any $\lambda(x_0)\in \pi^{-1}(x_0)$, we consider the distribution $\lambda(x_0)\delta_{x_0}\in \mathcal{D}'(M,V)$, given by
\[
(\lambda(x_0)\delta_{x_0}, \psi)_{L^2(M,V)}=\lambda(x_0)\delta_{x_0}(\psi d\mu) =\langle \lambda(x_0),\psi(x_0) \rangle_{x_0},\quad \psi\in C^\infty(M,V).
\]

The main point of the bundle reconstruction is to use the energy functions $\mathcal{E}^{(N)}_{V,A}$, $N=1,2,\dots$,  to find for any point $x_0\in M$, sequences of sources that converge to delta distributions $\lambda(x_0)\delta_{x_0}$, $\lambda(x_0)\in\pi^{-1}(x_0)$. 

In what follows, we set $s_0=n/2+1\in \mathbb{N}$, if $n$ is even, and $s_0=n/2+1/2\in \mathbb{N}$, if $n$ is odd, where $n$ is the dimension of the manifold $M$.   We have
\[
\delta_{x_0}\in H^{-s_0}(M,\R),\quad \partial^\alpha\delta_{x_0}\not\in  H^{-s_0}(M,\R), \quad |\alpha|\ge 1.
\]

Let us equip $M$ with  a smooth Riemannian metric, and let $B(x_0,r)$ be an open ball on $M$, centered at a point $x_0\in M$, with radius $r$, with respect to this metric.  
For any point $x_0\in M$, 
there is a sequence  $(U_{p_l})_{l=1}^\infty$, $(p_l)_{l=1}^\infty\in X$,  of the sets $U_{p_l}\in \mathcal{B}_M$ 
such that 
\begin{equation}
\label{eq_sec_5_2_00}
x_0\in U_{p_l}\subset B(x_0,1/l),\quad l=1,2,\dots.
\end{equation}
Let $q_0\in \N$ and recall that for each $p_l$, $l=1,2,\dots$,  the set of sources $\mathcal{F}_{V; U_{p_{l}},I_{q_0}}=\{F^{p_l,q_0}_i\in C^0(\overline{I_{q_0}},L^2(U_{p_l},V)):i=1,2,\dots\}$
is dense in $ L^2(I_{q_0}, L^2(U_{p_l},V))$.

Viewing the sources $F^{p_l,q_0}_{i}(t)$, $t\in I_{q_0}$, as elements of $H^{-s_0}(M,V)$, we define the set $\mathcal{L}^{(1)}_{x_0,q_0}$ as the set of all sequences $\big(p_l, i_l,t_l\big)_{l=1}^\infty$ such that $p_l\in \N$, $p_l$ satisfies \eqref{eq_sec_5_2_00},  $i_l\in \N$, $t_l\in I_{q_0}$, and for any $\varphi\in H^{s_0}(M,V)$, the limit 
\begin{equation}
\label{eq_sec_5_2}
\lim_{l\to\infty}(F^{p_l,q_0}_{i_l}(t_l),\varphi)_{L^2(M,V)}
\end{equation}
exists. This implies that the sequence $(F^{p_l,q_0}_{i_l}(t_l),\varphi)_{L^2(M,V)}$, $l=1,2,\dots$,  is bounded, for any $\varphi\in H^{s_0}(M,V)$, and therefore, by the Banach-Steinhaus theorem, it is uniformly bounded, i.e.
\[
|(F^{p_l,q_0}_{i_l}(t_l),\varphi)_{L^2(M,V)}|\le C \|\varphi\|_{H^{s_0}(M,V)}, 
\]
with $C>0$ independent of $\varphi$ and $l$.  
 Hence, the functional 
\[
G:H^{s_0}(M,V)\to \R, \quad G(\varphi)=\lim_{l\to\infty}(F^{p_l,q_0}_{i_l}(t_l),\varphi)_{L^2(M,V)},
\]  
is linear and bounded, and therefore $G\in H^{-s_0}(M,V)$. We have 
\[
\lim_{l\to \infty}F^{p_l,q_0}_{i_l}(t_l)=G
\] 
in the weak topology of $H^{-s_0}(M,V)$.  
Furthermore, $\supp(F^{p_l,q_0}_{i_l}(t_l))\subset \overline{U_{p_l}}\subset \overline{B(x_0,1/l)}$, and hence, $\supp(G)=\{x_0\}$.  Considering $G$ in a local trivialization near $x_0$,  it follows from 
\cite[Theorem 2.3.4]{Hor03} that 
$G$ is a finite linear combination of the delta distributions at $x_0$ and its derivatives with coefficients from $\pi^{-1}(x_0)$.  By the choice of $s_0$, the space $H^{-s_0}(M,\R)$ contains delta distributions while not their derivatives. Thus, we have $G=\lambda(x_0)\delta_{x_0}$ with some $\lambda(x_0)\in \pi^{-1}(x_0)$.

We obtain the following result. 

\begin{lem}
The sequence $\big(p_l, i_l,t_l\big)_{l=1}^\infty\in \mathcal{L}^{(1)}_{x_0,q_0}$ if and only if   
\[
\lim_{l\to\infty}F^{p_l,q_0}_{i_l}(t_l)=\lambda(x_0)\delta_{x_0},
\]
with some $\lambda(x_0)\in \pi^{-1}(x_0)$,  in the weak topology of $H^{-s_0}(M,V)$.

\end{lem}

\begin{lem}
\label{lem_non_empty_new}
We have   $\mathcal{L}^{(1)}_{x_0,q_0}\ne\emptyset$.
\end{lem}

\begin{proof}

Let $\lambda\in C^0(M,V)$ and for any $l=1,2,\dots$,  consider the following functions,
\[
G_l(x,t)=\lambda(x) \frac{1}{\text{vol}(U_{p_l})}\chi_{U_{p_l}}(x)\chi_{I_{q_0}}(t)\in L^2(I_{q_0}, L^2(U_{p_l},V)),
\]
where $\chi_{U_{p_l}}$ and $\chi_{I_{q_0}}$ are the characteristic functions of the sets $U_{p_l}$ and $I_{q_0}$, respectively, and
\[
\text{vol}(U_{p_l})=\int_{U_{p_l}} d\mu(x). 
\]
Let $t_l\in I_{q_0}$. Then for any $\varphi\in H^{s_0}(M,V)$, we get
\begin{equation}
\label{eq_sec_5_2_0}
\lim_{l\to \infty} (G_l(t_l),\varphi)_{L^2(M,V)}=\lim_{l\to \infty} \frac{1}{\text{vol}(U_{p_l})}\int_{U_{p_l}} \langle \lambda(x), \varphi(x)\rangle_x d\mu(x)= \langle \lambda(x_0), \varphi(x_0)\rangle_{x_0},
\end{equation}
since $s_0>n/2$ and by Sobolev's embedding theorem $H^{s_0}(M,V)\subset C^0(M,V)$.

As the set $\mathcal{F}_{V; U_{p_l},I_{q_0}}$ is dense in $L^2(I_{q_0}, L^2(U_{p_l}, V))$, for any $\varepsilon>0$, there is $i_l\in \N$ such that for the corresponding source $F^{p_l,q_0}_{i_l}\in \mathcal{F}_{V; U_{p_l},I_{q_0}}$, we have
\[
\|F^{p_l,q_0}_{i_l}-G_l\|_{L^2(I_{q_0}, L^2(U_{p_l}, V))}\le \varepsilon\sqrt{|I_{q_0}|}, 
\]
 where $|I_{q_0}|=\int_{I_{q_0}}dt$.   As the function $F^{p_l,q_0}_{i_l}$ is continuous in time, by the mean value theorem for integrals, there is $t_l\in I_{q_0}$ such that 
  \[
\|F^{p_l,q_0}_{i_l}(t_l)-G_l(t_l)\|_{L^2(U_{p_l}, V)}\le \varepsilon.  
\]
Therefore, for any $\varphi\in H^{s_0}(M,V)$, we have
\[
\lim_{l\to \infty}(F^{p_l,q_0}_{i_l}(t_l)-G_l(t_l),\varphi)_{L^2(M,V)}=0. 
\]
This together with \eqref{eq_sec_5_2_0} implies that 
\[
\lim_{l\to \infty} (F^{p_l,q_0}_{i_l}(t_l),\varphi)_{L^2(M,V)}=\langle \lambda(x_0), \varphi(x_0)\rangle_{x_0}.
\] 
Hence, the sequence $(p_l, i_l,t_l)_{l=1}^\infty\in \mathcal{L}^{(1)}_{x_0,q_0}$. The proof is complete. 

\end{proof}

In order to reconstruct the vector bundle $V$ we shall need the following determination result.

\begin{lem}
\label{lem_const_1_new}
Given the energy functions $\mathcal{E}^{(N)}_{V,A}$, $N=1,2,\dots$, we can construct the set 
$\mathcal{L}^{(1)}_{x_0,q_0}$. 
\end{lem}

\begin{proof} 
Let $q_1\in \N$ and $\tau_0\in \R$ be such that $\overline{I_{q_1}}<\overline{I_{q_0}}<\tau$. 
Then it follows from Lemma \ref{lem_wave_op-new} that 
\[
H^{s_0}(M,V)=\{\p_t u^{\hat F}(\tau_0): \hat F\in \overline{\mathcal{F}}^{s_0}_{q_1,\tau_0}\}, 
\]
and therefore, the condition \eqref{eq_sec_5_2} is equivalent to the fact that for all $\hat F\in \overline{\mathcal{F}}^{s_0}_{q_1,\tau_0}$, the limit
\begin{equation}
\label{eq_sec_5_3}
\lim_{l\to\infty}(F^{p_l,q_0}_{i_l}(t_l),\p_t u^{\hat F}(\tau_0))_{L^2(M,V)}
\end{equation}
 exists. 
 Corollary \ref{cor_measurements_derivative} together with \eqref{eq_sec_5_1} implies that using the energy functions $\mathcal{E}^{(N)}_{V,A}$,  we can compute the inner products, 
\[
(F^{p_l,q_0}_{i_l}(t_l), \p_t u^{\hat F}(\tau_0))_{L^2(M,V)}
\]
for any $\hat F\in \overline{\mathcal{F}}^{s_0}_{q_1,\tau_0}$. Hence,  for each sequence $\big(p_l, i_l,t_l\big)_{l=1}^\infty$, we can check  whether the limit 
\eqref{eq_sec_5_3} exists  for any $\hat F\in \overline{\mathcal{F}}^{s_0}_{q_1,\tau_0}$. In such a way, we can construct the set  $\mathcal{L}^{(1)}_{x_0,q_0}$.  The proof is complete. 

\end{proof}

For $x_0\in M$, let $U_{\nu}\in \mathcal{B}_M$ be  small  and such that $x_0\in U_\nu$.  Then for any point $x\in U_{\nu}$, 
there is a sequence  $(U_{p_l(x)})_{l=1}^\infty$, $(p_l(x))_{l=1}^\infty\in X$,  of the sets $U_{p_l(x)}\in \mathcal{B}_M$ 
such that 
\begin{equation}
\label{eq_sec_6_1}
x\in U_{p_l(x)}\subset B(x,1/l),\quad l=1,2,\dots.
\end{equation}

For each such $p_l(x)$, $l=1,2,\dots$,  we consider the set of sources 
\[
\mathcal{F}_{V; U_{p_{l}(x)},I_{q_0}}=\{F^{p_l(x),q_0}_i\in C^0(\overline{I_{q_0}}, L^2(U_{p_l(x)},V)):i=1,2,\dots\},
\] 
which is dense in $L^2(I_{q_0}, L^2(U_{p_l(x)},V))$.

Let $\mathcal{L}^{(2)}_{U_{\nu}, q_0}$ be the set of all sequences of functions $\big(p_l(x),i_l(x), t_l(x)\big)_{l=1}^\infty$, $x\in U_{\nu}$, such that
\begin{itemize}
\item[1.] for any $x\in U_{\nu}$, $\big(p_l(x), i_l(x),t_l(x)\big)_{l=1}^\infty\in \mathcal{L}^{(1)}_{x,q_0}$, i.e. $p_l(x)\in \N$, $i_l(x)\in \N$, $t_l(x)\in I_{q_0}$ and  $\lim_{l\to\infty}F^{p_l(x),q_0}_{i_l(x)}(t_l(x))=\lambda(x)\delta_x$ in the weak topology of $H^{-s_0}(M,V)$.

\item[2.] $x\mapsto \lambda(x)$ is a $C^\infty$--smooth section in $U_{\nu}$.

\end{itemize}

The fact that $\mathcal{L}^{(2)}_{U_{\nu}, q_0}\ne\emptyset$  can be easily seen by the arguments, used in the proof of Lemma \ref{lem_non_empty_new}.

\begin{lem}
Given the energy functions $\mathcal{E}^{(N)}_{V,A}$, $N=1,2,\dots$, we can construct the set $\mathcal{L}^{(2)}_{U_{\nu}, q_0}$. 
\end{lem}

\begin{proof}

First by Lemma \ref{lem_const_1_new}, for any $x\in U_{\nu}$, we can construct the set $\mathcal{L}^{(1)}_{x,q_0}$.  Hence, for any $\big(p_l(x), i_l(x),t_l(x)\big)_{l=1}^\infty\in \mathcal{L}^{(1)}_{x,q_0}$, we get
\[
\lim_{l\to\infty}F^{p_l(x),q_0}_{i_l(x)}(t_l(x))=\lambda(x)\delta_x,
\]
with some $\lambda(x)\in \pi^{-1}(x)$, in the weak topology of $H^{-s_0}(M,V)$.   We have the local section, 
\[
\lambda: U_{\nu}\to V,\quad  x\mapsto \lambda(x).
\]  

We need to show that for any sequence $\big(p_l(x), i_l(x),t_l(x)\big)_{l=1}^\infty\in \mathcal{L}^{(1)}_{x,q_0}$, $x\in U_{\nu}$, using the energy functions $\mathcal{E}^{(N)}_{V,A}$, we can check whether $\lambda$ is $C^\infty$--smooth. 
It is clear that $\lambda$ is a $C^\infty$-smooth section in $U_{\nu}$ if and only if functions
\[
K_{\phi}:U_{\nu}\to \mathbb{R},\quad K(x)=\langle\lambda(x),\phi(x)\rangle_x
\]
are $C^\infty$--smooth for all $\phi\in C^\infty(M,V)$.

Let an interval $I_{q_1}\in\mathcal{B}_{\R}$ and $\tau_0\in\R$ be such that $\overline{I_{q_1}}<\overline{I_{q_0}}<\tau_0$.  Then 
\[
C^\infty(M,V)=\{\p_t u^{\hat F}(\tau_0):\hat F\in  \overline{\mathcal{F}}^\infty_{q_1,\tau_0}\}. 
\]
Hence, checking whether the functions $K_\phi$ are
$C^\infty$ smooth for all $\phi\in C^\infty(M,V)$ is equivalent to checking that the functions
\begin{align*}
\langle\lambda(x),\p_t u^{\hat F}(x,\tau_0)\rangle_{x}&=
(\lambda(x)\delta_{x},\p_t u^{\hat F}(\tau_0))_{L^2(U_{\nu},V)}\\
&=
\lim_{l\to\infty} ( F^{p_l(x),q_0}_{i_l(x)}(t_l(x)),\p_t u^{\hat F}(\tau_0))_{L^2(U_{\nu},V)}
\end{align*}
are $C^\infty$-smooth  for any generalized source $\hat F\in \overline{\mathcal{F}}^\infty_{q_1,\tau_0}$.  
Corollary \ref{cor_measurements_derivative} implies that the latter can be checked  using the energy functions $\mathcal{E}^{(N)}_{V,A}$. The proof is complete. 

\end{proof}

Notice that in the definition of the set $\mathcal{L}^{(2)}_{U_{\nu},q_0}$ the sets $U_{p_l(x)}\in \mathcal{B}_M$ such that 
\eqref{eq_sec_6_1} holds are fixed.  This means that every sequence of functions from $\mathcal{L}^{(2)}_{U_{\nu},q_0}$ has the first element  given by $p_l(x)$, satisfying \eqref{eq_sec_6_1}.

For $d=1,2,\dots$, 
let  $\mathcal{L}^{(3,d)}_{U_{\nu},q_0}$ be  the set consisting of collections of $d$ sequences of functions,
\[
\big(p_l(x), i_l^1(x), t_l^1(x)\big)_{l=1}^\infty, \dots, \big(p_l(x), i_l^d(x), t_l^d(x)\big)_{l=1}^\infty,\quad  x\in U_{\nu},
\]
 such that
\begin{itemize}
\item [1.] $\big(p_l(x), i_l^k(x), t_l^k(x)\big)_{l=1}^\infty\in \mathcal{L}^{(2)}_{U_{\nu}, q_0}$, $k=1,\dots,d$, i.e.
\[
\lim_{l\to\infty}F^{p_l(x),q_0}_{i^k_l(x)}(t^k_l(x))=\lambda^k(x)\delta_x, 
\]
in the weak topology of $H^{-s_0}(M,V)$, with a $C^\infty$--smooth section $\lambda^k$, $k=1,\dots, d$;
\item [2.] $\lambda^1(x),\dots,\lambda^d(x)\in\pi^{-1}(x)$ are linearly independent for each $x\in U_\nu$;

\end{itemize}

\begin{lem}
Given the energy functions $\mathcal{E}^{(N)}_{V,A}$, $N=1,2,\dots$, for each $d=1,2,\dots$, we can construct the set $\mathcal{L}^{(3,d)}_{U_{\nu},q_0}$ or show that it is empty. 

\end{lem}

\begin{proof}
Let $d=1,2,\dots$ be fixed. Then we need to show that 
for each collection of the $d$ sequences of functions 
$\big(p_l(x), i_l^1(x), t_l^1(x)\big)_{l=1}^\infty, \dots, \big(p_l(x), i_l^d(x), t_l^d(x)\big)_{l=1}^\infty\in \mathcal{L}^{(2)}_{U_{\nu}, q_0}$, using the energy functions, we can determine  whether the corresponding vectors $\lambda^1(x),\dots,\lambda^d(x)\in\pi^{-1}(x)$ are linearly independent  at $x\in U_\nu$. Indeed, 
for each $x\in U_\nu$, it is obvious that
$\lambda^1(x),\dots,\lambda^d(x)\in\pi^{-1}(x)$ are linearly independent if and only if the functionals
\[
K_{\lambda^k}:\pi^{-1}(x)\to\mathbb{R},\quad K_{\lambda^k}(p(x))=\langle\lambda^k(x),p(x)\rangle_{x},\quad k=1,2,\dots,d,
\]
are linearly independent.
Let an interval $I_{q_1}\in\mathcal{B}_{\R}$ and $\tau_0\in \R$ be such that $\overline{I_{q_1}}<\overline{I_{q_0}}<\tau_0$. Then by the construction of the generalized sources,
\begin{align*}
\pi^{-1}(x)&=\{p(x):p:M\to V\ \text{is a }C^\infty \text{--smooth section}\}\\
&=\{\p_t u^{\hat F}(x,\tau_0):\hat F\in\overline{\mathcal{F}}^\infty_{q_1,\tau_0}\}.
\end{align*}
Now the linear independence of $K_{\lambda^k}$, $k=1,\dots,d$, is equivalent to the linear independence of the functionals
\begin{align*}
K_k:\overline{\mathcal{F}}^\infty_{q_1,\tau_0}\to\mathbb{R}, K_k(\hat F)&=\langle\lambda^k(x),\p_t u^{\hat F}(x,\tau_0)\rangle_{x}
=(\lambda^k(x)\delta_{x},\p_t u^{\hat F}(\tau_0))_{L^2(M,V)}\\
&=\lim_{l\to\infty}(F^{p_l(x),q_0}_{i_l^k(x)}(t_l^k(x)),\p_tu^{\hat F}(\tau_0))_{L^2(M,V)},\ k=1,\dots,d,
\end{align*}
which in turn can be checked thanks to 
Corollary \ref{cor_measurements_derivative}.  The proof is complete.

\end{proof}

Let $U_\nu\in \mathcal{B}_M$, $x_0\in U_\nu$, be such that there is  a number $d\in\N$ for which $\mathcal{L}^{(3,d)}_{U_{\nu},q_0}\ne\emptyset$.  The existence of such $U_\nu$ and $d$ follows from the arguments similar to those in the proof of  Lemma \ref{lem_non_empty_new}.  
Let now $d_\nu$ be the maximum number among all $d$ such that $\mathcal{L}^{(3,d)}_{U_{\nu},q_0}\ne\emptyset$. 
We take $U_{\mu}$  such that for any $U_\nu\subset U_{\mu}$, $x_0\in U_\nu\in \mathcal{B}_M$, $d_{\nu}=d_{\mu}:=d_0$. 
In such a way we obtain the rank $d_0$ the vector bundle $V$, which we have to reconstruct. 

Taking a finite open cover of $M$, consisting of sets of the form $U_\mu$, for each such set, we have obtained the sequences of functions $\big(p_l(x), i_l^k(x), t_l^k(x)\big)_{l=1}^\infty$, $k=1,\dots,d_0$, such that 
\[
\lim_{l\to\infty}F^{p_l(x),q_0}_{i^k_l(x)}(t^k_l(x))=\lambda^k_\mu(x)\delta_x, 
\]
in the weak topology of $H^{-s_0}(M,V)$, with $\lambda^1_\mu, \dots,\lambda^{d_0}_\mu$ forming a basis for $V$ over $U_\mu$.  Here $F^{p_l(x),q_0}_{i^k_l(x)}(t^k_l(x))\in \mathcal{F}_{V; U_{p_l(x)},I_{q_0}}$, and $t^k_l(x)\in I_{q_0}$.

Our next step is to determine the inner product in the fibers $\pi^{-1}(x)$, $x\in M$, of $V$. 

\begin{lem}
\label{lem_inner_new}
Given the energy functions $\mathcal{E}^{(N)}_{V,A}$, $N=1,2,\dots$, we can determine 
the inner products $\langle \lambda^1(x), \lambda^2(x)\rangle_x$ for any $\lambda^1(x),\lambda^2(x)\in \pi^{-1}(x)$, $x\in M$. 

\end{lem}

\begin{proof}

Let $I_{q_1}\in \mathcal{B}_{\R}$ and $\tau_0\in\R$ be such that $\overline{I_{q_1}}<\overline{I_{q_0}}<\tau_0$. 
Consider the set
\[
\mathcal{S}=\{\hat{F}\in \overline{\mathcal{F}}^\infty_{q_1,\tau_0}: u^{\hat F}(x,\tau_0)=0 \quad \text{for all} \ x\in M\}.
\]
We have $\mathcal{S}\ne\{0\}$, since the wave operator  $W^\infty$, introduced in \eqref{eq_W-infty},  is bijective.

Given the energy functions $\mathcal{E}^{(N)}_{V,A}$, $N=1,2,\dots$,  the set
$\mathcal{S}$ can be determined. Indeed, let $\lambda^1_{\mu},\dots,\lambda^{d_0}_{\mu}$ be a basis for $V$ over $U_\mu$. Then by  
Corollary \ref{cor_inner_f}, given the energy functions, for any $\hat{F}\in \overline{\mathcal{F}}^\infty_{q_1,\tau_0}$, we can check whether 
\[
\langle Au^{\hat F}(x,\tau_0),\lambda^k_\mu(x)\rangle_x=\lim_{l\to\infty}(Au^{\hat F}(x,\tau_0), F^{p_l(x),q_0}_{i^k_l(x)}(t^k_l(x)))_{L^2(M,V)}=0,
\]
for $k=1,\dots, d_0$ and for any $x\in U_\mu$. Doing this for any $U_\mu$ from the open cover of $M$, we can check whether $Au^{\hat F}(x,\tau_0)=0$ for all $x\in M$. As the operator $A$ is positive, the fact that $Au^{\hat F}(x,\tau_0)=0$ is equivalent to the fact that $u^{\hat F}(x,\tau_0)=0$.

Let
 $x_0\in M$ be an arbitrary point and let $U_\mu$ be a set from the open cover of $M$ such that $x_0\in U_\mu$. In order to recover the inner product in the fiber $\pi^{-1}(x_0)$ we shall use the fact that 
 \[
 \pi^{-1}(x_0)=\{\lambda(x_0):\lambda \textrm{ is a } C^\infty\textrm{ smooth section}\}=\{\p_t u^{\hat F}(x_0,\tau_0): \hat F\in \mathcal{S}\}. 
 \]
 We have
 \begin{equation}
\label{eq_inner}
\langle \partial_t u^{\hat F}(x_0,\tau_0),\partial_tu^{\hat F}(x_0,\tau_0)\rangle_{ x_0}^2=\lim_{r\to 0}\frac{\|\partial_t u^{\hat F}(\tau_0)\|^2_{L^2(B(x_0,r),V)}}{\text{vol}(B(x_0,r))},
\end{equation}
where $B(x_0,r)\subset U_\mu$.  

As the manifold $M$ and the density $d\mu$ are known, we can find $\text{vol}(B(x_0,r))$.
Let us explain how to determine the $L^2$-norm  of the wave $\partial_t u^{\hat F}(\tau_0)$ in the ball $B(x_0,r)$.  
First of all,  since $\hat F\in \mathcal{S}$, by the definition of the energy, we can find the  $L^2$-norm of the wave $\partial_t u^{\hat F}(\tau_0)$ on the whole manifold $M$,
\[
\|\partial_t u^{\hat F}(\tau_0)\|_{L^2(M,V)}^2=2E_A (\hat F,\tau_0).
\]
Let
\[
K=K(\hat F,x_0,r)=\{\hat H\in\mathcal{S}:\partial_t u^{\hat H}(x,\tau_0)=\partial_t u^{\hat F}(x,\tau_0)\quad \text{for all } x\in B( x_0,r)\}.
\]

Given the energy functions $\mathcal{E}^{(N)}_{V,A}$, $N=1,2,\dots$,  the set $K$ can be determined. Indeed let $\lambda^1_\mu,\dots,\lambda^{d_0}_\mu$ be a basis for $V$ over $U_{\mu}$.  Then $\hat H\in K$ if and only if
\[
\langle\partial_t u^{ \hat F- \hat H}(x,\tau_0),\lambda^k_\mu(x)\rangle_x=0,\quad \text{for all}\ x\in B(x_0,r),\ k=1,\dots,d_0.
\]
By Corollary \ref{cor_measurements_derivative} for any $\hat H\in\mathcal{S}$,  the above conditions can be verified using the energy functions.

Hence, we can compute 
\[
\|\partial_t u^{\hat F}(\tau_0)\|^2_{L^2(B(x_0,r),V)}=\inf_{\hat{H}\in K}\|\partial_t u^{\hat H}(\tau_0)\|_{L^2(M,V)}^2,
\]
and therefore, using \eqref{eq_inner}, for any $\hat F\in\mathcal{S}$,  we can determine the inner product 
$\langle \partial_t u^{\hat F}(x_0,\tau_0),\partial_tu^{\hat F}(x_0,\tau_0)\rangle_{ x_0}$. 

Furthermore, using the polarization formula for a real vector bundle, for any $\hat F,\hat H\in\mathcal{S}$, we get
\[
\langle\partial_t u^{\hat F}(x_0,\tau_0) ,\partial_t u^{\hat H}(x_0,\tau_0)\rangle_{x_0}=\frac{1}{4}(
 \|\partial_t u^{\hat F+\hat H}(x_0,\tau_0) \|_{ x_0}^2- \|\partial_t u^{\hat F-\hat H}(x_0,\tau_0)  \|_{x_0}^2 ).
\]
The proof is complete. 

\end{proof}

Our next step is to reconstruct local trivializations of the vector bundle $V$. 

\begin{lem}
\label{thm_geometry2}
Given the energy functions $\mathcal{E}^{(N)}_{V,A}$, $N=1,2,\dots$, we can reconstruct local trivializations $\phi_{\mu}:\pi^{-1}(U_\mu)\to U_\mu\times\mathbb{R}^{d_0}$ of the vector bundle $V$ and
the $GL(d,\R)$-cocycle $\{t_{\mu\nu}\}$, $t_{\mu\nu}:U_\mu\cap U_\nu\to \text{GL}(d_0,\mathbb{R})$ between any two local trivializations.
\end{lem}

\begin{proof}

Let $U_\mu$ be  a set from the open cover of $M$, constructed above, and let $\lambda^1_{\mu},\dots,\lambda^{d_0}_{\mu}$ be a basis for $V$ over $U_\mu$.  Let $I_{q_1}\in \mathcal{B}_{\R}$ and $\tau_0\in\R$ be such that $\overline{I_{q_1}}<\overline{I_{q_0}}<\tau_0$.
Then for any $x\in U_\mu$, 
\[
\pi^{-1}(x)=\{\p_t u^{\hat F}(x,\tau_0):\hat F\in \overline{\mathcal{F}}^\infty_{q_1,\tau_0}\}. 
\]
We have 
\[
\p_t u^{\hat F}(x,\tau_0)=\sum_{k=1}^{d_0}\alpha^k_\mu(x)\lambda_\mu^k(x),\quad x\in U_\mu. 
\]
Given the energy functions, we can determine the coefficients $\alpha^k_\mu(x)$, $x\in U_\mu$, $k=1,\dots,d_0$, uniquely. Indeed, we have the following system for $\alpha^k_\mu$, 
\begin{equation}
\label{eq_sec_6_2}
\langle \p_t u^{\hat F}(x,\tau_0), \lambda^j_\mu(x)\rangle_x=\sum_{k=1}^{d_0}\alpha^k_\mu(x)\langle \lambda_\mu^k(x), \lambda^j_\mu(x)\rangle_x,\quad x\in U_\mu, \quad j=1,\dots,d_0. 
\end{equation}
By Corollary \ref{cor_measurements_derivative} and Lemma \ref{lem_inner_new}, the left hand side of the system \eqref{eq_sec_6_2} and the inner products $\langle \lambda_\mu^k(x), \lambda^j_\mu(x)\rangle_x$  can be computed using the energy functions.  As  $\lambda^1_{\mu},\dots,\lambda^{d_0}_{\mu}$ is a basis for $V$ over $U_\mu$, the system \eqref{eq_sec_6_2}  is uniquely solvable.

Thus, we can define a local trivialization $\phi_\mu:\pi^{-1}( U_\mu)\to U_\mu\times\mathbb{R}^{d_0}$ of the vector bundle $V$ as follows: 
\[
\phi_\mu(\p_t u^{\hat F}(x,\tau_0))=(x,\alpha^1_\mu(x),\dots,\alpha^{d_0}_\mu(x)), \quad x\in U_\mu, \quad \hat F\in\overline{\mathcal{F}}^\infty_{q_1,\tau_0}.
\]

 Let $U_\mu\cap U_\nu\ne\emptyset$ and  let $\lambda^1_\mu,\dots,\lambda^{d_0}_\mu$  and $\lambda^1_\nu,\dots,\lambda^{d_0}_\nu$ be bases for $V$ over $U_\mu$ and $U_\nu$, respectively. Then for any $x\in U_\mu\cap U_\nu$, we have
\[
u^{\hat F}(x,\tau_0)=\sum_{k=1}^{d_0}\alpha^k_\mu(x)\lambda_\mu^k(x)=\sum_{k=1}^{d_0}\alpha^k_\nu(x)\lambda_\nu^k(x).
\]
Therefore, one can find a matrix $G(x)\in \textrm{GL}(d_0,\mathbb{R})$ such that
\[
(\alpha_\mu^1(x),\dots, \alpha_\mu^{d_0}(x))=(\alpha_\nu^1(x),\dots, \alpha_\nu^{d_0}(x))G(x),
\]
and hence,
the $GL(d_0,\R)$-cocycle $t_{\mu\nu}(x)=G(x)$. The proof is complete.

\end{proof}

Finally, thanks to Theorem  \ref{thm_geometry1}, the vector bundle $V$ can be determined up to an isometry.

\subsection{Reconstruction of the operator $A$}

In order to determine the operator $A$  it is enough to find its representation in an arbitrary local trivialization $\phi_{\mu}:\pi^{-1}(U_\mu)\to U_\mu\times \mathbb{R}^{d_0}$. Recall that
\[
C^\infty(U_\mu,\R^{d_0})=\{\p_t u^{\hat F}(x,\tau_0)|_{U_\mu}:\hat F\in \mathcal{F}^\infty_{q_1,\tau_0}\}. 
\]
 Using the wave equation, we get
\[
A\p_t u^{\hat F}(x,\tau_0)=-\partial_t^3 u^{\hat F}(x,\tau_0),\quad x\in M,
\]
as $\overline{I_{q_1}}<\tau_0$. Thus, by Corollary \ref{cor_measurements_deriv_s},  we can find the representations of $\p_t u^{\hat F}(x,\tau_0)$ and $\partial_t^3 u^{\hat F}(x,\tau_0)$ in the local trivialization $\phi_\mu$, and therefore,  the graph of the operator $A$ in the local trivialization $\phi_\mu$, 
\[
\{(\p_t u^{\hat F}(\tau_0)|_{U_\mu}, -u^{\partial_t^3\hat F}(x,\tau_0)|_{U_\mu}):\hat F\in \overline{\mathcal{F}}^\infty_{q_1,\tau_0}\}.
\]

The proof of Theorem \ref{thm_main} is complete.

\section{Generation of the data for the inverse problem. Random sources}
\label{sec_random_sources}

\def\expec{{\mathbb E \,}}
\def\prob{{\mathbb P \,}}

In Theorem \ref{thm_main} we assume that for any  interval $I\in \mathcal{B}_{\R}$ and  any  set $U\in\mathcal{B}_M$,  there is a countable set of sources $\mathcal{F}_{V;U,I}\subset C^0(\overline{I},L^2(U,V))$, which is dense in $L^2(I,L^2(U,V))$,  and a countable set $\mathcal{F}_{V; M,I}\subset C_0^\infty(\overline{I},C^\infty(M,V))$, which is dense in  $C_0^\infty(\overline{I},C^\infty(M,V))$.

The purpose of this section is to show that this assumption is generic in the sense that the sets  $\mathcal{F}_{V;U,I}$ and $\mathcal{F}_{V; M,I}$  can be almost surely generated by taking some sequences of realizations of suitable independent identically distributed Gaussian random variables.

We shall start by recalling some basic notions of probability theory, following \cite{Bog1998, Hairer}.  Let $(\Omega,\Sigma,\mathbb{P})$ be a complete probability space, and let $\mathcal{H}$ be a real separable Hilbert space with the inner product $(\cdot,\cdot)_{\mathcal{H}}$.  We shall identify the dual of $\mathcal{H}$ with $\mathcal{H}$, using the Riesz representation theorem.

A measurable map $X:(\Omega,\Sigma)\to (\mathcal{H},\mathcal{B}(\mathcal{H}))$ is said to be an $\mathcal{H}$--valued random variable. Here $\mathcal{B}(\mathcal{H})$ is the Borel $\sigma$-algebra of $\mathcal{H}$ with respect to  the norm topology.  The probability law of $X$ is a Borel measure $\mu_X$ in $\mathcal{H}$ defined by
\[
\mu_X(B)=\mathbb{P}(X^{-1}(B)), \quad B\in \mathcal{B}(\mathcal{H}). 
\]
The random variable $X$ has the expectation $ \expec X\in \mathcal{H}$, given by
\[
 (\expec X, \varphi)_{\mathcal{H}}= \expec(X,\varphi)_{\mathcal{H}},\quad \varphi\in \mathcal{H},
\]
and the covariance operator $C_X:\mathcal{H}\to \mathcal{H}$, defined by
\begin{equation}
\label{eq_cov_op_new}
(C_X\varphi,\psi)_{\mathcal{H}}= \expec((X- \expec X, \varphi)_{\mathcal{H}}(X- \expec X,\psi)_{\mathcal{H}}),\quad \varphi,\psi\in \mathcal{H}. 
\end{equation}

We say that $X$ is a Gaussian random variable if for $\varphi\in \mathcal{H}$, the real-valued random variable $\omega\mapsto (X(\omega),\varphi)_{\mathcal{H}}$ is Gaussian.  The probability law $\mu_X$
 of the Gaussian random variable $X$ is a Gaussian probability measure on $\mathcal{H}$ in the sense that $l_*\mu$ is a Gaussian probability measure on $\R$ for every linear continuous functional $l:\mathcal{H}\to \R$. 

Given a Gaussian random variable $X$ with values in $\mathcal{H}$, it is known that its covariance operator $C_X$ is a non-negative, self-adjoint trace class operator in $\mathcal{H}$. Conversely, every non-negative self-adjoint trace class operator is a covariance operator of  a Gaussian random variable $X$ with values in $\mathcal{H}$, see \cite[Theorem 2.3.1]{Bog1998}. 

We shall need the following result. 

\begin{prop}
\label{prop_density_new}
Let $X_j$, $j=1,2,\dots$, be independent identically distributed Gaussian $\mathcal{H}$-valued random variables with $\mathbb{E}X_j=0$ and covariance operators $C_{X_j}:\mathcal{H}\to \mathcal{H}$ being injective. Then the set $\{X_j(\omega):j=1,2,\dots\}$ is almost surely dense in $\mathcal{H}$, with respect to the norm topology. 
\end{prop}

\begin{proof}

Let $X$ be an $\mathcal{H}$-valued Gaussian random variable with $\mathbb{E}(X)=0$ and an injective covariance operator $C_X$. Let $\mu_X$ be a probability law of $X$.  Associated with the Gaussian measure $\mu_X$
is the Cameron-Martin space $\mathcal{H}_{\mu_X}\subset \mathcal{H}$, defined as the completion of $\textrm{Ran}(C_X)$ with respect to the norm 
\[
\|x\|_{\mu_X}^2:=(C_Xx^*,x^*)_{\mathcal{H}},\quad x=C_X x^*. 
\]
The Cameron-Martin space $\mathcal{H}_{\mu_X}$ can also be characterized
as $\mathcal{H}_{\mu_X}=C_X^{1/2}(\mathcal{H})$.
Since the operator $C_X=C_X^*$ is injective,  $\textrm{Ran}(C_X)$ is dense in $\mathcal{H}$ with respect to the norm topology, and therefore, so is  the Cameron-Martin space $\mathcal{H}_{\mu_X}$. 

On the other hand, according to  \cite[Theorem 3.6.1]{Bog1998}, see also \cite[Section 3]{Hairer}, the support of the Gaussian measure $\mu_X$ is the closure of $\mathcal{H}_{\mu_X}$ in $\mathcal{H}$ with respect to the norm topology.
We recall that the support of $\mu_X$ consists of those points $x\in \mathcal{H}$ such that $\mu_X(U)>0$ for any neighborhood $U$ of $x$. 

Thus, for any non-empty open set $W\subset \mathcal{H}$, we have
\[
\mathbb{P}(X\in W)>0. 
\]

It is now easy to finish the proof. Since the space $\mathcal{H}$ is separable, it has a countable basis with respect to the norm topology, which we denote by $\{W_k\}_{k=1}^\infty$. 
As $X_j$, $j=1,2,\dots$, are independent identically distributed, we get for any $k=1,2,\dots$, 
\[
\mathbb{P}(\cup_{j=1}^\infty (X_j\in W_k) )=1,
\quad\text{i.e.} \quad \mathbb{P}([\cup_{j=1}^\infty (X_j\in W_k) ]^c)=0.
\] 
Hence, we obtain that 
\[
0=\sum_{k=1}^\infty\mathbb{P}([\cup_{j=1}^\infty (X_j\in W_k) ]^c)
\ge \mathbb{P}(\cup_{k=1}^\infty[\cup_{j=1}^\infty (X_j\in W_k) ]^c),
\]
and therefore, 
\[
1=\mathbb{P}(\cap_{k=1}^\infty  \cup_{j=1}^\infty (X_j\in W_k) )=
\mathbb{P}(\{\forall k\in\N \  \existsÊj\in \N: X_j\in W_k\}).
\]
Thus, 
\[
\mathbb{P}( \{\textrm{for any non-empty open set }W,\   \existsÊj\in \N: X_j\in W\} )=1,
\]
which completes the proof. 

\end{proof}

Our purpose now is construct a sequence of Gaussian random variables, satisfying the assumptions  of Proposition \ref{prop_density_new} taking values in a suitable scale of  Hilbert spaces. 
To that end, let $A:C^\infty(M,V)\to C^\infty(M,V)$ be an elliptic formally self-adjoint positive second order partial differential operator, as in Theorem \ref{thm_main}.  
We define the operator
\[
P=A+P_0, \quad P_0=-\p_t^2+t^2,
\]
on $L^2(\R,L^2(M,V))$.   The harmonic oscillator   $P_0$, equipped with the domain 
\[
\mathcal{D}(P_0)=\{u\in L^2(\R):P_0u\in L^2(\R)\}=\{u\in L^2(\R):t^j\p_t^k u\in L^2(\R), j+k\le 2\},
\]
is self-adjoint on $L^2(\R)$ with 
$\textrm{spec}(P_0)=\{2j+1:j=0,1,2,\dots\}$.  The domain of $P_0^m$ is given by
\[
\mathcal{D}(P_0^m)=\{u\in L^2(\R):t^j\p_t^k u\in L^2(\R), j+k\le 2m\},
\]
and furthermore,
\[
\bigcap_{m=1}^\infty \mathcal{D}(P_0^m)=\mathcal{S}(\R),
\]
where $\mathcal{S}(\R)$ is the Schwartz space, 
see \cite{Hitrik_Pra_2009}.  The operator $A$, equipped with the domain 
\[
\mathcal{D}(A)=\{u\in L^2(M,V): A u\in L^2(M,V)\}=H^2(M,V),
\]
is self-adjoint positive on $L^2(M,V)$. Furthermore,
\[
\mathcal{D}(A^m)=H^{2m}(M,V),\quad \bigcap_{m=1}^\infty\mathcal{D}(A^m)=C^\infty(M,V). 
\]

The operator $P$, equipped with the domain $\mathcal{D}(P)=\{u\in L^2(\R, L^2(M, V)): Pu\in L^2(\R, L^2(M,V))\}$,  is self-adjoint with the discrete spectrum, given by 
\[
\textrm{spec}(P)=\textrm{spec}(A)+ \textrm{spec}(P_0).
\] 
Let $0<\lambda_1\le \lambda_2\le\dots$ be the eigenvalues of $P$, repeated according to their multiplicity, and $\varphi_j\in \mathcal{S}(\R_t)\otimes C^\infty(M,V)$ be the  corresponding orthonormal basis of eigenfunctions, $P\varphi_j=\lambda_j\varphi_j$, $j=1,2,\dots$.  

We consider next a rough upper bound for the counting function $N(E)$ of the eigenvalues of $P$, 
\[
N(E)=\sum_{\mu+\nu\le E}1.
\] 
Here $ \mu\in\textrm{spec}(A)$ and $\nu\in  \textrm{spec}(P_0)$. We have
\[
N(E)=\sum_{\nu\le E}\sum_{\mu\le E-\nu}1=\sum_{\nu\le E}N_A(E-\nu)\le c\sum_{\nu\le E}(E-\nu)^{n/2}\le cE^{n/2+1},
\]
where $N_A$ is the counting function for the eigenvalues of $A$, and we have used that $N_A(s)\le cs^{n/2}$, see \cite{Shubin_book}.
It follows that 
\begin{equation}
\label{eq_sec_10_1}
\lambda_j\ge j^{\frac{2}{n+2}}/c,\quad c>0,\quad j=1,2,\dots.
\end{equation}

Let $C_X=e^{-P}$. Then $C_X$ is non-negative self-adjoint operator on $L^2(\R,L^2(M,V))$ with the eigenvalues of the form $e^{-\lambda_j}$, $j=1,2,\dots$. In particular, we see that $C_X$ is injective, and \eqref{eq_sec_10_1} implies that $C_X$ is of trace class. 

We shall now introduce explicitly a Gaussian random variable $X$ for which $C_X$ is the covariance operator. When doing so, let $Y_j\in N(0,1)$, $j=1,2,\dots$, be a collection of independent identically distributed real-valued Gaussian random variables. We define
\begin{equation}
\label{eq_sec_10_1 bb}
X=\sum_{j=1}^\infty  e^{-\lambda_j/2}Y_j\varphi_j.
\end{equation}
As the operator $C_X$ is of trace class,  formula \eqref{eq_sec_10_1 bb} defines a $L^2(\R, L^2(M,V))$-valued random variable, see e.g.\ \cite[Thm.\ 2.3.1]{Bog1998}.

Furthermore, one easily sees that $\mathbb{E}X=0$ and the covariance operator of $X$ is exactly $C_X$. 

Next we shall show that $X\in  \bigcap_{m=1}^\infty\mathcal{D}(P^m)=\mathcal{S}(\R_t)\otimes C^\infty(M,V)$ almost surely.   Indeed, for  $m=1,2,\dots$, we have
\[
P^mX=\sum_{j=1}^\infty  e^{-\lambda_j/2}Y_j\lambda_j^m\varphi_j,
\]
and 
\[
\mathbb{E}(\|P^mX\|^2_{L^2(\R, L^2(M, V))})=\sum_{j=1}^\infty  e^{-\lambda_j}\mathbb{E}(|Y_j|^2)\lambda_j^{2m}=\sum_{j=1}^\infty  e^{-\lambda_j}\lambda_j^{2m}<\infty.
\]
Here we have used \eqref{eq_sec_10_1}.

Viewing the random variable $X$ as taking values in the Hilbert space $\mathcal{D}(P^m)$, using \eqref{eq_cov_op_new}, we may check that the corresponding covariance  operator is given by
\[
C_{X,\mathcal{D}(P^m)}=P^{-m}C_{P^mX}P^m:\mathcal{D}(P^m)\to \mathcal{D}(P^m),
\]
where $C_{P^mX}:L^2(\R, L^2(M,V))\to L^2(\R,L^2(M,V))$ is the covariance operator of the random variable $P^mX$, taking values in $L^2(\R, L^2(M,V))$. Since the operator $C_{P^mX}$ is injective, the operator $C_{X,\mathcal{D}(P^m)}$ is injective as well. 

We summarize the discussion above in the following proposition. 

\begin{prop}
\label{prop_gauss_new}
There exists a Gaussian random variable $X$ with values in $L^2(\R,L^2(M,V))$ such that 
\begin{itemize}
\item[(i)]  $X\in \mathcal{S}(\R_t)\otimes C^\infty(M,V)$ almost surely; 

\item[(ii)] $\mathbb{E}(X)=0$;

\item[(iii)] the operators $C_X:L^2(\R, L^2(M,V))\to L^2(\R, L^2(M,V))$ and $C_{X,\mathcal{D}(P^m)}:\mathcal{D}(P^m)\to \mathcal{D}(P^m)$, $m=1,2,\dots$, are injective.

\end{itemize}

\end{prop}

Combining Proposition \ref{prop_density_new} and  Proposition \ref{prop_gauss_new}, we obtain a sequence $X_j$, $j=1,2,\dots$, of  independent identically distributed Gaussian random variables taking values in  $\mathcal{S}(\R_t)\otimes C^\infty(M,V)$ such that the set $\{X_j(\omega): j=1,2,\dots\}$ is almost surely dense in $\mathcal{D}(P^m)$, with respect to the norm topology, for all $m=0, 1,2,\dots$. 

It follows that for any $U\in\mathcal{B}_M$ and $I\in \mathcal{B}_\R$, the set $\{X_j(\omega)|_{U\times I}: j=1,2,\dots\}\subset C^\infty(\overline{I}, C^\infty(\overline{U},V))$ is almost surely dense in $L^2(I, L^2(U,V))$.  

Finally, let $I\in \mathcal{B}_\R$. Using the set $\{X_j(\omega): j=1,2,\dots\}$, we shall construct a countable subset of $C^\infty_0(\overline{I}, C^\infty(M,V))$, which is almost surely dense in the latter space. 
To that end, notice that the set 
 $\{X_j(\omega):j=1,2,\dots\}$ is almost surely dense in  $\mathcal{S}(\R_t)\otimes C^\infty(M,V)$ with respect to the topology, generated by the seminorms 
 \[
 u\mapsto\|P^mu\|_{L^2(\R, L^2(M,V))},\quad m=0,1,2,\dots.
 \] 
 Let $\psi\in C^\infty_0(\overline{I})$ be such that $\psi\ge 0$ and $\psi(t)>0$ for all $t\in I$.  Then we have
 \begin{equation}
 \label{eq_fact}
 C_0^\infty(I)=\psi C^\infty_0(I),
 \end{equation}
 where $C_0^\infty(I)=\{\varphi\in C^\infty(\R):\supp(\varphi)\subset I\}$. Indeed, if $\varphi\in C^\infty_0(I)$, there is a smooth factorization,
 \[
 \varphi=\psi\frac{\varphi}{\psi},\quad \frac{\varphi}{\psi}\in C_0^\infty(I). 
 \]
 On the other hand, using that elements of the space $C^\infty_0(\overline{I})$ vanish to infinite order at the end points of $I$, we see that $C^\infty_0(I)$ is dense in $C_0^\infty(\overline{I})$,  with respect to the Fr\'echet topology, generated by the seminorms $\varphi\mapsto \sup_{t\in \overline{I}}|\varphi^{(j)}(t)|$, $j=0,1,2,\dots$.

 As $C^\infty_0(I, C^\infty(M,V))\subset \mathcal{S}(\R_t)\otimes C^\infty(M,V)$, we know that the set $\{X_j(\omega): j=1,2,\dots\}$ is almost surely dense in $C^\infty_0(I, C^\infty(M,V))$ with respect to the topology for $C^\infty_0(\overline{I}, C^\infty(M,V))$.
Now using \eqref{eq_fact}, we conclude  that the set  $\{\psi X_j(\omega): j=1,2,\dots\}\subset C^\infty_0(\overline{I}, C^\infty(M,V))$ is almost surely dense in $C^\infty_0(\overline{I}, C^\infty(M,V))$. 
 Thus, one can almost surely obtain a dense set in $C^\infty_0(\overline{I}, C^\infty(M,V))$ by taking
 a sequence of realizations of the random variables $\psi X_j(\omega)$, $ j=1,2,\dots$

Summarizing the discussion in this section, we see that the assumption in Theorem \ref{thm_main} concerning the existence of countable dense sets of sources is generic in the precise sense described  above.

\section{Acknowledgements}

We are very grateful to Herbert Koch for the very careful reading of the manuscript,  numerous comments and suggestions, leading to substantial improvements in the presentation.  We are also very grateful to Boris Tsirelson for helping us with the proof of Proposition \ref{prop_density_new}.  The research of K.K. is supported by the
Academy of Finland (project 125599) and the research of M.L. is supported by the Academy of Finland Center of Excellence programme 213476 and project 141104.

\end{document}